\documentclass[a4paper,12pt]{scrartcl}

\usepackage[utf8]{inputenc}

\usepackage{amsthm,amsmath,amsfonts,amssymb}
\numberwithin{equation}{section} % numerazione delle equazioni

\usepackage{mathrsfs}

\usepackage{tikz-cd}

\usepackage{mathtools}	% per definire nuovi operatori
\usepackage{bm}	% simboli in gassetto
\usepackage{bbm}	% simboli in gassetto

\usepackage{tensor}	% index notation

\usepackage[colorlinks=true]{hyperref}	% riferimenti cliccabili
\usepackage[all]{hypcap}      % needed to help hyperlinks direct correctly;

\newcommand{\m}[1]{\mathcal{#1}}
\newcommand{\bb}[1]{\mathbb{#1}}

\newcommand{\mrm}[1]{\mathrm{#1}}

\newcommand{\diff}{\partial}
\newcommand{\bdiff}{\bar{\partial}} 

\newcommand{\I}{\mathrm{i}\mkern1mu}	% square root of -1
\newcommand{\transpose}{\intercal}		% matrix transpose

\DeclarePairedDelimiter\abs{\lvert}{\rvert}	% absolute value
\DeclarePairedDelimiter\norm{\lVert}{\rVert}	% norm

\DeclarePairedDelimiter{\set}{\{}{\}}	% per definire insiemi
\newcommand{\tc}{\mathrel{}\mathclose{}\middle|\mathopen{}\mathrel{}}	% asta verticale da usare con \set*

\newcommand{\cotJ}{T^*\!\m{J}}	% cotangente della varietà delle strutture complesse

\newcommand{\SpecRad}{\mathrm{r}\mkern1mu} % spectral radius

\newtheorem{theorem}{Theorem}[section]

\newtheorem{proposition}[theorem]{Proposition}
\newtheorem{lemma}[theorem]{Lemma}
\newtheorem{corollary}[theorem]{Corollary}

\theoremstyle{definition}
\newtheorem{definition}[theorem]{Definition}

\theoremstyle{remark}
\newtheorem{remark}[theorem]{Remark}
\newtheorem*{claim}{Claim}
\begin{document}
\title{The HcscK equations in symplectic coordinates}
\author{Carlo Scarpa         \and
        Jacopo Stoppa %etc.
}
%
%\institute{C. Scarpa \at
%SISSA, via Bonomea 265, 34136 Trieste, Italy\\
%ORCID iD \href{https://orcid.org/0000-0001-5218-6845}{0000-0001-5218-6845}\\
%\email{cscarpa@sissa.it}\\
%\and J. Stoppa \at
%SISSA, via Bonomea 265, 34136 Trieste, Italy\\
%\email{jstoppa@sissa.it} 
%}
\date{\vspace{-10pt}}

\maketitle

\begin{abstract}
\noindent\textsc{Abstract.} The Donaldson-Fujiki K\"ahler reduction of the space of compatible almost complex structures, leading to the interpretation of the scalar curvature of K\"ahler metrics as a moment map, can be lifted canonically to a hyperk\"ahler reduction. Donaldson proposed to consider the corresponding vanishing moment map conditions as (fully nonlinear) analogues of Hitchin's equations, for which the underlying bundle is replaced by a polarised manifold. However this construction is well understood only in the case of complex curves. In this paper we study Donaldson's hyperk\"ahler reduction on abelian varieties and toric manifolds. We obtain a decoupling result, a variational characterisation, a relation to $K$-stability in the toric case, and prove existence and uniqueness under suitable assumptions on the ``Higgs tensor''. We also discuss some aspects of the analogy with Higgs bundles.\newline
%\keywords{canonical metrics in K\"ahler geometry \and hyperk\"ahler geometry \and toric manifolds \and abelian varieties}
\end{abstract}

\setcounter{secnumdepth}{2}
\setcounter{tocdepth}{2}
\tableofcontents

\section{Introduction}

Let~$(M,\omega)$ denote a compact K\"ahler manifold of dimension~$n$, with a fixed K\"ahler form. Donaldson~\cite{Donaldson_scalar} and Fujiki~\cite{Fujiki_moduli} constructed a Hamiltonian action of the group of Hamiltonian symplectomorphisms~$\operatorname{Ham}(M, \omega)$ on the space~$\mathcal{J}$ of almost complex structures compatible with the symplectic form~$\omega$, endowed with a natural K\"ahler structure. It turns out that the moment map for this action, evaluated at~$J\in\m{J}$ on Hamiltonians~$h$, is given by
\begin{equation*}
\mu_J(h) = \int_M 2 (s(g_J) -\hat{s}) h \,\frac{\omega^n}{n!},
\end{equation*}
where~$s(g_J)$ denotes the Hermitian scalar curvature of the Hermitian metric~$g_J = \omega\circ J$. Thus the zero moment map equation is precisely the constant scalar curvature condition for~$g_J$. This fact found important applications in complex differential geometry (see e.g. the classic~\cite{Donaldson_scalarcurvature_embeddingsI}). 

Donaldson~\cite{Donaldson_hyperkahler} also proposed to study the induced~$\operatorname{Ham}(M, \omega)$-action on the cotangent bundle~$\cotJ$. By analogy with the finite-dimensional case, one expects that this carries the structure of an infinite-dimensio\-nal hyperk\"ahler manifold, preserved by the action of~$\operatorname{Ham}(M, \omega)$. Then the corresponding real and complex moment map equations would give a close analogue of Hitchin's equations for harmonic bundles (originally introduced in~\cite{Hitchin_self_duality}), in the context of special metrics in K\"ahler geometry. In this analogy~$\cotJ$ replaces the cotangent bundle to the space of~$\bdiff$-operators on a fixed smooth Hermitian bundle, and~$\operatorname{Ham}(M, \omega)$ plays the role of the unitary gauge group. So one can think of these equations roughly as deforming the constant scalar curvature condition with a ``Higgs field''. Note that in this analogy the rank of the bundle corresponds to the dimension of~$M$. The case of ``line bundles", i.e. complex curves, was studied in detail in~\cite{Donaldson_hyperkahler,Hodge_phd_thesis,ScarpaStoppa_HcscK_curve} (see also~\cite{traut} for related results). 

In the higher dimensional case the relevant hyperk\"ahler structure on~$\cotJ$ was described explicitly in~\cite{ScarpaStoppa_hyperk_reduction}. Let us briefly recall the construction. As in the work of Donaldson and Fujiki, one regards~$\m{J}$ as the space of sections of a (non-principal)~$\operatorname{Sp}(2n)$-bundle with fibres diffeomorphic to the model space of compatible linear complex structures, namely the quotient~$\operatorname{Sp}(2n)/\operatorname{U}(n)$. This is a Hermitian symmetric space of noncompact type, naturally identified with Siegel's upper half space.

By general results due to Biquard and Gauduchon~\cite{Biquard_Gauduchon} there exists a \emph{unique}~$\operatorname{Sp}(2n)$-invariant real function~$\rho$, defined on a~$\operatorname{Sp}(2n)$-invariant open neighbourhood of the zero section of the cotangent bundle~$T^*\!(\operatorname{Sp}(2n)/\operatorname{U}(n))$, such that adding~$\I\diff\bdiff\rho$ to the pullback of the invariant K\"ahler form on~$\operatorname{Sp}(2n)/\operatorname{U}(n)$ we obtain a hyperk\"ahler metric. Since the cotangent bundle~$\cotJ$ is a~$\operatorname{Sp}(2n)$-bundle with fibres diffeomorphic to~$T^*(\operatorname{Sp}(2n)/\operatorname{U}(n))$, this can be transferred to an infinite-dimensional, formally hyperk\"ahler metric on a~$\operatorname{Ham}(M, \omega)$-invariant open neighbourhood~$\mathcal{U}$ of the zero section of~$\cotJ$, preserved by the action of~$\operatorname{Ham}(M, \omega)$: this is the content of Theorem 1.1 in~\cite{ScarpaStoppa_hyperk_reduction}.

Let~$\Omega_{\operatorname{I}}$ denote the K\"ahler form on~$\mathcal{U}$ corresponding to the hyperk\"ahler metric and the standard complex structure~$\operatorname{I}$ on~$\cotJ$ (induced by the Donaldson-Fujiki complex structure on~$\mathcal{J}$). Let~$\Theta$ be the tautological holomorphic symplectic form on~$\cotJ$. It is shown in~\cite{ScarpaStoppa_hyperk_reduction} Theorem 1.2 that the action of~$\operatorname{Ham}(M, \omega)$ on~$\mathcal{U}$ is Hamiltonian with respect to both the real symplectic form~$\Omega_{\operatorname{I}}$ and the complex symplectic form~$\Theta$. The corresponding real and complex moment maps, evaluated at a point~$(J, \alpha) \in \mathcal{U} \subset \cotJ$ ($\alpha$ denoting a morphism~$\alpha:{T^{0,1}}^*\!M\to{T^{1,0}}^*\!M$, with dual~$\alpha^{\vee}$), are given respectively by 
\begin{equation}\label{eq:real_mm_implicit}
m^{\bb{R}}_{(J,\alpha)}(h)=\mu_{J}(h) + \int_{x\in M}d^c\rho_{(J(x),\alpha(x))}\left(\mathcal{L}_{X_h}J,\mathcal{L}_{X_h}\alpha\right)\,\frac{\omega^n}{n!}
\end{equation}
and
\begin{equation*}
m^{\bb{C}}_{(J,\alpha)}(h)= -\int_M\frac{1}{2}\mrm{Tr}(\alpha^{\vee}\mathcal{L}_{X_h}J)\,\frac{\omega^n}{n!}.
\end{equation*}
The corresponding vanishing moment map conditions are the \emph{coupled} equations, to be solved for~$(J, \alpha)$,
\begin{equation}\label{eq:HcscK}
m^{\bb{R}}_{(J,\alpha)}(h) = m^{\bb{C}}_{(J,\alpha)}(h) = 0, \, \forall h \in \mathfrak{ham}(M, \omega).
\end{equation}
Because of their close relation to the constant scalar curvature condition, and the analogy with Hitchin's harmonic bundle equations, these are called the \emph{Hitchin-cscK (HcscK) equations}  in~\cite{ScarpaStoppa_hyperk_reduction}.

\medbreak

In the present paper we study the equations~\eqref{eq:HcscK} in the special cases when~$M$ is a complex torus or a toric manifold. The two situations share an interesting feature, the existence of a set of \emph{symplectic coordinates} on an open dense subset of~$M$.

\subsection{The abelian case}

Consider first the case of a complex torus. Without loss of generality we can assume that this is in fact the abelian variety
\begin{equation*}
M=\mathbb{C}^n/(\mathbb{Z}^n+i\mathbb{Z}^n),
\end{equation*} 
endowed with the standard flat K\"ahler form given in complex coordinates~$z = x + i w$ by
\begin{equation*}
\omega = i \sum_{a}dz^a\wedge d\bar{z}^a.
\end{equation*}
Even on this simple geometry the equations~\eqref{eq:HcscK} are highly nontrivial. Note that there is a real torus~$\mathbb{T}^n \subset \operatorname{Symp}(M, \omega)$, with~$\bb{T}^n \cong \bb{R}^n/\bb{Z}^n$, acting on~$M$ by translations   
\begin{equation*}
t\cdot(x +i w)=x+i(w+t).
\end{equation*} 
We will study a particular set of solutions~$(J, \alpha)$ to the HcscK equations~\eqref{eq:HcscK} on~$M$ such that the~$\omega$-compatible almost complex structure~$J$ and the endomorphism~$\alpha$ (giving a cotangent vector at~$J$) are both~$\mathbb{T}^n$-invariant. The class of almost complex structures~$J$ we consider is that of Legendre duals to K\"ahler potentials for invariant K\"ahler metrics in the class~$[\omega]$. Such a~$J$ is automatically integrable, but we make no integrability assumption on the deformation of complex structure~$\alpha^{\vee}$. This class of Legendre dual almost complex structures is standard in toric complex differential geometry and can be understood formally as an orbit of the (non-existent) complexification of the ``gauge group''~$\operatorname{Ham}(M, \omega)$; we refer the reader to~\cite{Donaldson_scalar} for more details.

Under these assumptions, there exist real coordinates~$(y, w)$ such that 
\begin{equation}\label{eq:LegendreDualJ}
J(y, w) = J(y) = 
\left(\begin{matrix} 0 & -(D^2 u)^{-1} \\
D^2 u & 0
\end{matrix}\right),
\end{equation}
where~$u(y)$ is a convex function on~$\bb{R}^n$, with periodic Hessian, of the form
\begin{equation*}
u(y) = \frac{1}{2}\abs{y}^2 + \phi(y).
\end{equation*}
Note that 
%\begin{equation*}
$u_0 = \frac{1}{2}\abs{y}^2$
%\end{equation*}
corresponds to our fixed reference complex structure~$J_0$ on~$M$, with flat K\"ahler metric~$g_{J_0}$. A cotangent vector~$\alpha$ at such~$J \in \m{J}$ can be regarded as an endomorphism of the (trivial) complexified cotangent bundle~$T^*_{\bb{C}}M$. We can lower an index of~$\alpha$ using the Hermitian metric~$g_J$ and obtain a complex bilinear form~$\xi$ on the fibres of the (trivial) bundle~$T_{\bb{C}}M$. The fact that~$\alpha$ preserves~$\omega$-compatibility to first order (since it is a cotangent vector) translates into the symmetry condition 
%\begin{equation*}
$\xi^{ab} = \xi^{ba}$.
%\end{equation*} 
Thus in effect both our relevant tensors~$D^2 u$ and~$\xi$ are functions on the real torus~$\mathbb{T}^n$ with values in complex symmetric matrices; moreover~$D^2 u$ is real and positive definite. Pairs~$(u, \xi)$ corresponding to~$(J, \alpha) \in \mathcal{U} \subset \cotJ$, the admissible neighbourhood of the zero section, must satisfy the bound
\begin{equation*}
\SpecRad(u, \xi) < 1
\end{equation*}
on the spectral radius~$\SpecRad(u, \xi):= \SpecRad(\xi\,D^2u\,\bar{\xi}\,D^2u)$. In what follows we refer occasionally to~$u$ as the \emph{symplectic potential}. Abusing terminology, sometimes we also call~$\xi$ the \emph{Higgs tensor}, although \emph{no integrability assumption is made on~$\xi$}.

Our first result shows that the equations~\eqref{eq:HcscK} become more explicit and treatable when expressed in terms of~$u$ and~$\xi$. The real moment map equation is formulated in terms of the square root of the endomorphism
\begin{equation*}
\mathbbm{1}-\xi\,D^2u\,\bar{\xi}\,D^2u.
\end{equation*}
This square root exists and is unique: in the proof of Theorem~\ref{thm:real_mm_expl} we will show that~$\xi\,D^2u\,\bar{\xi}\,D^2u$ is similar to a Hermitian endomorphism with eigenvalues in the interval~$[0,1)$.
\begin{theorem}\label{thm:HcscKSymplCoordThm}
The HcscK equations~\eqref{eq:HcscK} on the abelian variety~$M$ for a~$\mathbb{T}^n$-invariant symplectic form~$\omega$, complex structure~$J$ and deformation~$\alpha$, with~$J$ given by~\eqref{eq:LegendreDualJ}, are equivalent to the system of \emph{uncoupled} partial differential equations  
\begin{equation}\label{eq:HcscK_squareroot}
\begin{dcases}
\xi^{ab}_{,ab}=0\\
\left(\left(\mathbbm{1}-\xi\,D^2u\,\bar{\xi}\,D^2u\right)^{\frac{1}{2}}D^2u^{-1}\right)^{ab}_{,ab}=0,
\end{dcases}
\end{equation}
corresponding the vanishing of the complex and real moment map respectively.
\end{theorem}
\begin{remark}
We will show (Proposition~\ref{proposition:integrability_xi}) that the integrability condition for the deformation of complex structure~$\alpha^{\vee}$ dual to the cotangent vector~$\alpha$ is equivalent to~$\xi$ being of the form~$(D^2u)^{-1}\,D^2\varphi\,(D^2u)^{-1}$, for a function~$\varphi$ whose Hessian is a periodic function of~$y$. Thus a general solution of the complex moment map equation does not correspond to an integrable deformation of the complex structure, but rather to an almost K\"ahler deformation.
\end{remark}
\begin{remark}\label{U1Rmk}
By~\eqref{eq:HcscK_squareroot} the~$\operatorname{U}(1)$-action on~$T^*\m{J}$ given by~$\theta\cdot\alpha = e^{i\theta}\alpha$ preserves the real and complex moment map equations.
\end{remark}

We will show that the real moment map equation also admits a useful variational characterisation. In order to describe this, let us introduce the strictly increasing, convex function of a real variable~$x\in [0,1]$ given by 
\begin{equation*}
f(x) =1-(1-x)^{\frac{1}{2}}+ \log(1+(1-x)^{\frac{1}{2}}).
\end{equation*}
The aforementioned Biquard-Gauduchon function on~$\cotJ$ can then be expressed as a spectral function of the endomorphism~$\alpha\bar{\alpha}$ of~$T^*M$ (see the proof of Theorem~\ref{thm:real_mm_expl}):
\begin{equation*}
\rho(J,\alpha) = \mrm{Tr}\,f(\alpha\bar{\alpha}).
\end{equation*}
Let us write~$d\mu$ for the Lebesgue measure on the real torus~$\mathbb{T}^n$. We define the \emph{Biquard-Gauduchon functional} as
\begin{equation}\label{eq:BGFunc}
\m{H}(u, \xi) = \frac{1}{2} \int_{\mathbb{T}^n} \rho\left(\xi\,D^2u\,\bar{\xi}\,D^2u\right)d\mu.
\end{equation}
The \emph{periodic~$K$-energy} (see~\cite[\S$2$]{FengSzekelyhidi_abelian}) is given by
\begin{equation}\label{eq:KEn}
\m{F}(u) = -\frac{1}{2} \int_{\mathbb{T}^n} \log\det(D^2 u)  d\mu.
\end{equation}
Then we define the \emph{periodic \textit{HK}-energy} as
\begin{equation}\label{eq:HKEn}
\widehat{\m{F}}(u, \xi) = \m{F}(u) + \m{H}(u,\xi).
\end{equation}
\begin{proposition}\label{proposition:HKEulLagr}
For fixed Higgs tensor~$\xi$, the real moment map equation in~\eqref{eq:HcscK_squareroot} is the Euler-Lagrange equation, with respect to variations of the symplectic potential~$u$, for the periodic \textit{HK}-energy~$\widehat{\m{F}}(u, \xi)$.
\end{proposition}  
Our main application of this variational characterisation is a uniqueness result, which relies in turn on a key convexity property.

\begin{theorem}\label{thm:convexity}
Fix a Higgs tensor~$\xi$. The periodic \textit{HK}-energy, regarded as a functional on symplectic potentials, is convex along linear paths in 
\begin{equation*}
\mathcal{A}(\xi)=\set*{u\mbox{ symplectic potential}\tc\SpecRad(\xi\,D^2u\,\bar{\xi}\,D^2u)<1}.
\end{equation*}
\end{theorem}

\begin{remark}\label{remark:A_convex}
Note that in general the set~$\mathcal{A}(\xi)$ is not convex. However if we assume that the real and imaginary parts of~$\xi$ are positive or negative semidefinite, then the subset
\begin{equation*} 
\mathcal{A}'(\xi) \subset \mathcal{A}(\xi)
\end{equation*}
consisting of all symplectic potentials~$u$ such that
\begin{equation*}
\norm{\xi}^2_u=\operatorname{Tr}(\xi\,D^2u\,\bar{\xi}\,D^2u)<1
\end{equation*}
is convex. It is enough to show this when the real and imaginary parts of~$\xi$ are definite, and then this follows from the fact that with our assumption~$\norm{\xi}^2_{u}$ is a convex function, along a linear path~$u_t$ in~$\m{A}(\xi)$. Indeed we have  
\begin{align*}
&\frac{d^2}{dt^2}\norm{\xi}^2_{u_t} = 2\mrm{Tr}(\xi D^2\dot{u}_t\,\bar{\xi}D^2\dot{u}_t).  
\end{align*}
If~$\Re(\xi)$ is definite, then~$\Re(\xi) D^2\dot{u}_t$ is similar to a symmetric matrix and so it has real spectrum, thus
\begin{equation*}
\mrm{Tr}(\Re(\xi) D^2\dot{u}_t \Re(\xi) D^2\dot{u}_t) = \mrm{Tr}(\Re(\xi) D^2\dot{u}_t)^2 \geq 0.
\end{equation*}
The same inequality holds for the imaginary part~$\Im(\xi)$, and both conditions together imply~$\mrm{Tr}(\xi D^2\dot{u}_t\,\bar{\xi}D^2\dot{u}_t) \geq 0$. 

A similar computation shows that~$\mathcal{A}'(\xi)$ is also convex in the special case when~$M$ is a surface and~$\det(\xi) = 0$. We will use this fact in the proofs of Theorems~\ref{HcscK1dThm} and~\ref{LowRankThm}. 
\end{remark}

\begin{corollary} If the real and imaginary parts of~$\xi$ are semidefinite, then the extremal points~$u$ of the \textit{HK}-energy in~$\mathcal{A}'(\xi)$ are minima, and are in fact unique up to constants.
\end{corollary}

\begin{remark}
It is important to point our that, even under positivity assumptions on~$\xi$, the Biquard-Gauduchon functional~$\m{H}(u, \xi)$ is not, in general, convex with respect to~$u$. The crucial point is that, as we will show, the periodic~$K$-energy~$\m{F}(u)$ compensates this lack of convexity.
\end{remark}

We now turn to existence results for the system~\eqref{eq:HcscK_squareroot}. First, the complex moment map equation~\eqref{eq:HcscK_squareroot} has already been studied by Donaldson (see~\cite[\S$7.2$]{Donaldson_scalar}) in a related context.
\begin{lemma}[\cite{Donaldson_scalar}]\label{lemma:complex_mm_sol}
On any convex domain in~$\bb{R}^n$, the general solution of the equation~$\xi^{ab}_{,ab}=0$ is given by
\begin{equation*}
\xi^{ab}=\diff_cT^{abc}+\diff_cT^{bac}
\end{equation*}
for any tensor~$T^{abc}$ that is anti-symmetric in~$b$,~$c$.
\end{lemma}
We can use Lemma~\ref{lemma:complex_mm_sol} to obtain periodic solutions of the complex moment map equation for the system~\eqref{eq:HcscK_squareroot}: it is sufficient to choose a tensor field~$T^{abc}$ for which~$\diff_cT^{abc}$ is periodic. 

In order to formulate an existence result for the real moment map equation, recall that~$\norm{\xi}_{u_0}$ denotes the pointwise norm of the Higgs tensor with respect to the flat K\"ahler metric~$g_{J_0}$.
\begin{theorem}\label{QuantPertThm} There is a constant~$K > 0$, depending only on well-known elliptic estimates on the real torus~$\mathbb{T}^n$ with respect to the flat metric~$g_{J_0}$, such that if~$\xi$ satisfies
%\begin{equation*}
$\norm{\xi}_{u_0} < K$
%\end{equation*} 
then there exists a solution~$u$ to the real moment map equation in~\eqref{eq:HcscK_squareroot}. If the real and imaginary part of~$\xi$ are semidefinite then~$u$ is unique up to an additive constant.
\end{theorem}
More precisely~$K$ depends only on standard Schauder estimates, on estimates for the linearised Monge-Amp\`ere equations, such as those in~\cite{TrudingerWang_MongeAmpere} Section 3.7, Corollary 3.2, and on Caffarelli's H\"older estimates for the real Monge-Amp\`ere equation. So Theorem~\ref{QuantPertThm} can be seen as a quantitative perturbation result around the locus~$\xi \equiv 0$, for which the unique solution is given by the flat metric~$u_0$, up to an additive constant. Of course this result is still rather implicit, and one would like to obtain a more concrete estimate. We do not know how to achieve this in general, but we discuss in detail the special case on a~$2$-dimensional torus when the Higgs tensor~$\xi$ depends only on a single variable, say~$y^1$, and satisfies~$\det(\xi) = 0$. We show that in this case the real moment map equation is ``integrable'', i.e. it reduces to an algebraic condition, and use this to prove an effective existence result.
\begin{theorem}\label{HcscK1dThm}
On the abelian surface~$\mathbb{C}^2/(\mathbb{Z}^2+i\mathbb{Z}^2)$, suppose the Higgs tensor~$\xi = \xi(y^1)$ solves the complex moment map equation and does not have maximal rank. Then either 
\begin{equation*}
\xi = \begin{pmatrix}
0 & 0\\ 0& \xi^{22}(y^1)
\end{pmatrix},
\end{equation*}
in which case the real moment map equation has the unique solution~$u = u_0$ (up to an additive constant), or 
\begin{equation*}
\xi = \begin{pmatrix} c & \xi^{12}(y^1)\\ \xi^{12}(y^1) & \frac{\left(\xi^{12}(y^1)\right)^2}{c}\end{pmatrix},\, c \in \bb{C}^*
\end{equation*}
and there is a unique solution~$u(y)$ up to constants to the real moment map equation provided 
\begin{equation*}
|\xi^{12}(y^1)|\leq |c|<\frac{3}{10}.
\end{equation*} 
In both cases~$\xi$ corresponds to an integrable deformation if and only if it is constant.  
\end{theorem}
Finally we consider further the analogy between the equations~\eqref{eq:HcscK_squareroot} and the classical Hitchin equations for harmonic bundles~\cite{Hitchin_self_duality}, in the two-dimensional case. There are at least two features which carry over to the general situation of the HcscK equations on a complex surface. These are conveniently expressed in terms of the endomorphism~$A = \alpha^{\vee} = A^{1,0} + A^{0,1}$ of~$T_{\bb{C}}M$ dual to the cotangent vector~$\alpha$. 

On the one hand, there is an action of~$\operatorname{U}(1)$ on~$\cotJ$, dual to~$\theta \cdot A^{1,0} = e^{i \theta} A^{1,0}$, and this preserves the real and complex moment map equations, see Remark~\ref{U1Rmk}. In the special case of Theorem~\ref{thm:HcscKSymplCoordThm}, this is clear: the action maps~$\xi$ to~$e^{i\theta}\xi$, and the equations~\eqref{eq:HcscK_squareroot} are manifestly invariant.

Secondly, we can consider the map which associates to~$A$ the elementary symmetric polynomials in its eigenvalues. In general, one shows that the only non vanishing polynomials are 
\begin{equation*}
\sigma_2(A) = -\frac{1}{2}\mrm{Tr} A^2 = \norm*{A^{0,1}}^2_{g_{J}},\,\sigma_4(A) = \det(A) =\abs*{\det(A^{0,1})}^2. 
\end{equation*}
For our special cotangent pairs corresponding to~$(u, \xi)$, these are the functions
\begin{equation*}
\sigma_2(A) =\norm*{\xi}^2_u,\,\sigma_4(A) = \abs*{\det(\xi D^2 u)}^2.  
\end{equation*}
In general, for~$\varphi \in \operatorname{Ham}(M, \omega)$, these quantities transform as
\begin{equation*}
\varphi^*\sigma_i(A) = \sigma_i(\varphi^* A), 
\end{equation*}
but, in our special case, if we restrict to~$\mathbb{T}^2$-invariant tensors and to Hamiltonian symplectomorphisms commuting with the action of~$\mathbb{T}^2$, then we have  
\begin{equation*}
\varphi^*\sigma_i(A) = \sigma_i(A), 
\end{equation*}
so that it makes sense to consider the ``gauge invariant'' map given by
\begin{equation*}
\pi(A) =\left(\norm*{A^{0,1 }}^2_{g_{J}},\abs*{\det(A^{0,1})}^2\right).
\end{equation*}
Thus the special locus~$\det(\xi) = 0$ shares some features with the ``global nilpotent cone'', the zero fibre of the Hitchin system: in particular, it is the locus where the Higgs tensor is most degenerate. We prove a first structure result for this locus.
\begin{theorem}\label{LowRankThm}
Let~$\xi_0\in\bb{C}^{2\times 2}$ be a constant symmetric matrix such that~$\det(\xi_0) = 0$ and~$\SpecRad(\xi_0\bar{\xi_0}) < 1$. Nearby~$(u_0, \xi_0)$ the locus of solutions~$(u, \xi)$ to the two-dimensional periodic HcscK system satisfying~$\det(\xi) = 0$ is an infinite dimensional submanifold in the space of all solutions, provided~$\xi_0$ does not belong to an exceptional subset with empty interior. 
\end{theorem}

\subsection{The toric case}

We now come to the second topic of this article, toric manifolds. Consider a K\"ahler manifold~$(M^n,J,\omega)$ with a Hamiltonian action~$\bb{T}^n\curvearrowright M$, whose moment map~$\mu$ sends~$M$ to a convex polytope~$P\subseteq\bb{R}^n$ by the Atiyah-Guillemin-Sternberg Theorem~\cite{Atiyah_convexity,Guillemin_convexity}. It is well-know that~$P$ is a \emph{Delzant polytope}~\cite{Delzant_polytope}, and that any such polytope defines in turn a compact symplectic manifold~$(M_P,\omega_P)$, together with a Hamiltonian~$\bb{T}^n$-action on~$M_P$ such that~$(M_P,\omega_P)$ is equivariantly isomorphic to~$(M,\omega)$. The polytope defines also a standard compatible complex structure~$J_P$, but in general~$(M_P,J_P,\omega_P)$ will not be isomorphic to~$(M,J,\omega)$. For the general theory we refer to~\cite{Guillemin_toric,Guillemin_momentmaps} and~\cite{Abreu_toric}.

The moment map gives an alternative way to describe the symplectic structure on~$M$, since it establishes a~$\bb{T}^n$-equivariant isomorphism between~$(M^\circ,\omega)$ and~$(P^\circ\times\bb{T}^n,\omega_P)$, where~$M^\circ$ is the open subset of~$M$ on which the action is free, and~$\omega_P$ is the standard symplectic structure induced by the inclusion in~$\bb{R}^{2n}$. On the other hand, the manifold~$M^\circ$ is~$\bb{T}^n$-equivariantly biholomorphic to~$\bb{C}^n/2\pi\I\bb{Z}^n\cong\bb{R}^n+\I\bb{T}^n$, where the action is by translations on the~$\bb{T}^n$-factor. We consider the standard coordinates~${z}={x}+\I{w}$ on~$\bb{R}^n\times\bb{T}^n$. On~$P^\circ\times\bb{T}^n$ instead we consider coordinates~$({y},{w})$, with~${y}=\mu({x})$.

The symplectic form on~$M^\circ$ is given by~$4\I\diff\bdiff v$ for some~$\bb{T}^n$-invariant potential function~$v$, so that~$\omega=\I\diff_{x^a}\diff_{x^b}v\,\mrm{d}z^a\wedge\mrm{d}\bar{z}^b$, while~$J$ is just represented by the canonical matrix~$J({x},{w})=\begin{pmatrix}0 & -\mathbbm{1} \\ \mathbbm{1} & 0\end{pmatrix}$. On the other hand the symplectic structure on~$P^\circ\times\bb{T}^n$ induced by~$\omega$ via the moment map is the canonical symplectic form~$\sum_{i,j}\mrm{d}y^i\wedge\mrm{d}w^j$ on~$\bb{R}^{2n}$, while the complex structure~$J$ is described by a matrix~$J({y},{w})=\begin{pmatrix} 0 & -G^{-1} \\ G & 0\end{pmatrix}$. Since~$J$ is integrable, this matrix must be of the form~$G=\mrm{Hess}_{{y}}u$ for some potential~$u({y})$. Moreover the two coordinate systems and the two functions~$u$ and~$v$ are Legendre dual to each other, that is they satisfy 
\begin{equation*}
{y}=\diff_{{x}}v,\, 
{x}=\diff_{{y}}u,\,
u({y})+v({x})={x}\cdot{y}.
\end{equation*}
This means that the HcscK system~\eqref{eq:HcscK} can be expressed in the coordinates~$({y},{w})$ in a form similar to what was done for abelian varieties in~\eqref{eq:HcscK_squareroot}. The only difference is that~$\hat{s}$ does not vanish, in general, for a toric manifold, so the system of equations in the interior of the moment polytope becomes, setting~$C=4\hat{s}$,
\begin{equation}\label{eq:symplectic_HcscK_toric}
\begin{cases}
\xi^{ab}_{,ab}=0\\
\left(\left(\mathbbm{1}-\xi\,D^2u\,\bar{\xi}\,D^2u\right)^{\frac{1}{2}}D^2u^{-1}\right)^{ab}_{,ab}=-C.
\end{cases}
\end{equation}
The system should be solved for a potential~$u$ and a deformation~$\xi$ of the complex structure, satisfying appropriate conditions at the boundary of~$P$. 
\begin{remark}
We emphasise again that we do \emph{not} require~$\xi$ to correspond to an integrable deformation (whose class in~$H^1(TM) \cong \{0\}$ would necessarily vanish), but only to an almost K\"ahler deformation. Indeed we will show that an integrable solution of the complex moment map equation necessarily vanishes, see Corollary~\ref{cor:toric_integrability}. 
\end{remark}
The boundary conditions for~$u$ are well-understood from the work of Abreu and will be reviewed in Section~\ref{SympCoordSec}. Briefly, the symplectic potential must be of the form~$u=u_P+h$ for~$h\in\m{C}^\infty(P)$ and a canonical potential~$u_P\in\m{C}^\infty(P^\circ)$ that is singular at the boundary of~$P$. We will show that the boundary conditions for~$\xi$ are also written in terms of~$u_P$.
\begin{proposition}\label{proposition:toric_boundarycond}
A Higgs tensor~$\xi$ on the polytope~$P$ extends to a deformation of the complex structure on~$M$ if and only if~$D^2u_P\,\xi\,D^2u_P$ extends smoothly to~$P$.
\end{proposition}
This condition on~$\xi$ allows us to get an integration by parts formula for the real moment map in~\eqref{eq:symplectic_HcscK_toric} on the Delzant polytope~$P$, in the spirit of~\cite[Lemma~$3.3.5$]{Donaldson_stability_toric}. This formula (see Lemma~\ref{lemma:intbyparts_realmm}) in turn allows us to get a variational characterization of the real moment map on toric manifolds. In order to write the relevant functional we introduce a measure~$\mrm{d}\sigma$ on the boundary of the polytope, which is the same defined by Donaldson in~\cite{Donaldson_stability_toric}. 
\begin{theorem}\label{thm:toricHKenergy}
Fix a Delzant polytope~$P$. For a a Higgs tensor~$\xi$, let~$\m{A}(\xi)$ be the space of symplectic potentials on~$P$ defined by
\begin{equation*}
\m{A}(\xi)=\set*{u=u_P+h\tc r(\xi\,D^2u\,\bar{\xi}\,D^2u)<1}.
\end{equation*}
The real moment map equation in~\eqref{eq:symplectic_HcscK_toric} is the Euler-Lagrange equation of the toric \textit{HK}-energy:
\begin{equation}\label{eq:toricHKenergy}
\begin{split}
\widehat{\m{F}}(u,\xi)=&\int_{\diff P}u\,\mrm{d}\sigma-\int_PC\,u\,\mrm{d}\mu-\int_P\log\mrm{det}\left(D^2u\right)\mrm{d}\mu\\&+\int_P\rho\left(\xi\,D^2u\,\bar{\xi}\,D^2u\right)\mrm{d}\mu.
\end{split}
\end{equation}
Moreover~$\widehat{\m{F}}(u,\xi)$ is convex along linear paths in~$\m{A}(\xi)$.
\end{theorem}
As in the periodic case, the convexity of the toric \textit{HK}-energy implies that solutions of the toric HcscK system in the subset~$\m{A}'(\xi) \subset \m{A}(\xi)$ are unique, if the real and imaginary parts of~$\xi$ are semidefinite (see Remark~\ref{remark:A_convex}).

The integration by parts formula in Lemma~\ref{lemma:intbyparts_realmm} also gives a first necessary condition for the existence of solutions of~\eqref{eq:symplectic_HcscK_toric}.
\begin{theorem}\label{thm:Kstab}
Assume that there is a solution~$(u,\xi)$ of the toric HcscK system. Then the polarised toric manifold~$M$ is uniformly~$K$-stable.
\end{theorem}
We refer to Section~\ref{sec:toric_HKenergy} for the definition of uniform~$K$-stability of a toric manifold.
\begin{remark}
Here and in the rest of the paper, we always consider \emph{$K$-stability with respect to toric test-configurations}, i.e. test-configurations whose total space is a toric variety, with a torus action covering that of~$M$. These test-configurations are described in terms of piecewise-linear convex functions on the polytope, we refer to~\cite{Donaldson_stability_toric} for the details. Toric~$K$-stability is a priori weaker than~$K$-stability, but the two notions are conjectured to coincide. Indeed, toric~$K$-stability should imply the existence of a cscK metric, which in turn implies general~$K$-stability by~\cite{BermanDarvasLu_weakminimizers}. For an algebraic result related to this conjecture, see~\cite{StoppaCodogni_torus}, where it is shown that a general test-configuration can be related to a~$\bb{T}$-equivariant filtration.
\end{remark}

Note that there is an analogy between Theorem \ref{thm:Kstab} and the case of Higgs bundles: it is possible that a rank~$2$ bundle~$V$ over a Riemann surface~$\Sigma$ supporting a solution of Hitchin's equations is necessarily stable. For example this happens when~$V$ has trivial determinant, the genus of~$\Sigma$ is~$2$, and~$V$ is not decomposable or an extension (see~\cite{Hitchin_self_duality} Example 3.13).

On toric manifolds, uniform~$K$-stability is equivalent to the existence of torus-invariant cscK metrics. For complex surfaces, this was shown by Donaldson in a series of papers culminating in \cite{Donaldson_cscKmetrics_toricsurfaces}. For higher-dimensional manifolds, this is a combination of \cite{Li_uniformKstab} and the famous results of Chen-Cheng \cite{ChenCheng_automorphisms} together with \cite{Hisamoto_stabilitytoric}. We refer to \cite{Apostolov_toric} for an overview of these results; see also \cite{ChiLi_geodesicrays_stability} for a different proof. Hence, Theorem \ref{thm:Kstab} implies that the existence of a torus-invariant cscK metric is a necessary condition for the existence of solutions to the HcscK system. We can try to deduce the converse by considering a perturbation around the cscK metric. In the $2$-dimensional case, relying on the results of~\cite{Li_toric_affinegeom} on the prescribed curvature problem on toric surfaces, we can obtain an analogue of Theorem~\ref{QuantPertThm}.
\begin{theorem}\label{QuantPertThmToric}
Let~$M$ be a uniformly K-stable polarised toric surface. There is a constant~$K > 0$, depending only on certain elliptic estimates on~$M$ endowed with the cscK metric~$g_{0}$, such that if~$\xi$ satisfies
%\begin{equation*}
$\norm{\xi}_{u_0} < K$
%\end{equation*} 
then there exists a solution~$u$ to the real moment map equation in~\eqref{eq:symplectic_HcscK_toric}. If the real and imaginary part of~$\xi$ are semidefinite then~$u$ is unique up to an additive constant.
\end{theorem}
It is natural to conjecture that a similar result should also hold on higher dimensional manifolds. We return on this point in Section \ref{sec:toric}.
\bigbreak

The paper is planned as follows. In Section~\ref{sec:real_mm} we write the real moment map in~\eqref{eq:HcscK} as the~$L^2$-pairing of~$h$ with some explicit function on~$M$ depending on~$J$,~$\omega$ and~$\alpha$. This result has independent interest and prepares our computations in Section~\ref{SympCoordSec}, which contains some background on K\"ahler geometry in symplectic coordinates, the proofs of Theorem~\ref{thm:HcscKSymplCoordThm} and Proposition~\ref{proposition:toric_boundarycond}, and a characterisation of the integrability condition for a deformation of the complex structure~$\alpha^{\vee}$. Section~\ref{ConvexitySec} proves our variational characterisation on abelian varieties and the convexity result, Proposition~\ref{proposition:HKEulLagr} and Theorem~\ref{thm:convexity}. Section~\ref{sec:toric_HKenergy} instead focuses on the variational characterization for toric manifolds and contains the proof of Theorem~\ref{thm:toricHKenergy} and Theorem~\ref{thm:Kstab}. The existence results for the periodic HcscK system, Theorems~\ref{QuantPertThm} and~\ref{HcscK1dThm}, are proved in Section~\ref{sec:periodic_sol}, together with Theorem~\ref{LowRankThm}. In Section~\ref{sec:toric} we investigate the existence of solutions to the toric HcscK system.

\paragraph{Acknowledgements.}
Part of this work was written while the first author was visiting the University of Illinois at Chicago. The first author wishes to thank the Department of Mathematics at UIC for kind hospitality, and particularly Julius Ross for many helpful discussions related to this work. We are grateful to Julien Keller and Vestislav Apostolov for some useful comments on an earlier version of these results. We are very grateful to the reviewer for the valuable suggestions, in particular regarding our stability result, Theorem \ref{thm:Kstab}. We would also like to thank all the participants in the K\"ahler geometry seminars at IGAP, Trieste.

\section{An explicit expression for the real moment map}\label{sec:real_mm}

The real moment map in~\eqref{eq:real_mm_implicit} is rather difficult to study in this implicit form. Using the~$L^2$-pairing of functions we can identify~$m^{\bb{R}}_{(J,\alpha)}$ with a smooth real function on~$M$. This has been shown in~\cite[\S$4$, \S$5$]{ScarpaStoppa_hyperk_reduction} in the special case of complex curves and surfaces. The aim of the present Section is to find an explicit expression for~$m^{\bb{R}}_{(J,\alpha)}\in\m{C}^\infty(M,\bb{R})$ in arbitrary dimension. This has independent interest. In Section~\ref{SympCoordSec} we will then find an alternative description of this function in symplectic coordinates.

\begin{theorem}\label{thm:real_mm_expl}
The real moment map equation in the HcscK system is equivalent to 
\begin{equation*}
2\left(s(g_J)-\hat{s}\right)+\operatorname{div}X(J,\alpha)=0
\end{equation*}
where the vector field~$X(J,\alpha)$ is written, in a system of complex coordinates, as
\begin{equation*}
X(J,\alpha)=2\,\Re\left(g(\nabla^a\alpha,\bar{\alpha}\hat{\alpha})\diff_{z^a}-g(\nabla^{\bar{b}}\alpha,\bar{\alpha}\hat{\alpha})\diff_{\bar{z}^b}-2\nabla^*(\alpha\bar{\alpha}\hat{\alpha})\right)
\end{equation*}
and~$\hat{\alpha}$ is the endomorphism of~$T^*M$ defined by
\begin{equation*}
\hat{\alpha}=\frac{1}{2}\left(\mathbbm{1}+\left(\mathbbm{1}-\alpha\bar{\alpha}\right)^{\frac{1}{2}}\right)^{-1}.
\end{equation*}
\end{theorem}

\begin{proof}
Recall that the real moment map~\eqref{eq:real_mm_implicit} is the sum of two terms,~$m^{\bb{R}}_{(J,\alpha)}(h)=\mu_{J}(h)+\mu'_{J,\alpha}(h)$, where~$\mu_{J}(h)=2\langle s(g_J)-\hat{s},h\rangle_{L^2(\omega)}$ and 
\begin{equation}\label{eq:real_mm_accessorio}
\mu'_{J,\alpha}(h)=\int_{x\in M}d^c\rho_{(J(x),\alpha(x))}\left(\mathcal{L}_{X_h}J,\mathcal{L}_{X_h}\alpha\right)\,\frac{\omega^n}{n!}.
\end{equation}
Our aim is to also write this integral as the~$L^2$-pairing of~$h$ with some function. The general calculations that are required have been carried out already in~\cite[\S$4.2$ and \S$5.2$]{ScarpaStoppa_hyperk_reduction}, however they depend on an explicit expression for the Biquard-Gauduchon function and its differential, which in~\cite{ScarpaStoppa_hyperk_reduction} was obtained only for complex dimension~$1$ or~$2$. We will show here how to remove this restriction. As shown in~\cite[\S$3.2$]{ScarpaStoppa_hyperk_reduction}, the Biquard-Gauduchon function~$\rho(J,\alpha)$ is computed as follows: let~$f$ be the spectral function
\begin{equation*}
f(x)=\frac{1}{x}\left(\left(1+x\right)^{\frac{1}{2}}-1-\mrm{log}\frac{1+\left(1+x\right)^{\frac{1}{2}}}{2}\right),
\end{equation*}
let~$A:=\alpha^\vee+\bar{\alpha}^\vee$ be the first-order deformation of the complex structure induced by~$\alpha$, and consider the endomorphism~$\Xi_A$ of~$T_J\m{J}$ defined by
\begin{equation*}
\Xi_A:B\mapsto-\frac{1}{2}\left(A^2B+BA^2\right).
\end{equation*}
Then the Biquard-Gauduchon function is
\begin{equation*}
\rho(J,\alpha)=\frac{1}{2}\mrm{Tr}\left(f(\Xi_A)(A)\,A\right).
\end{equation*}
In order to compute the spectral function~$f$ of the endomorphism~$\Xi_A$, we will first find a basis of~$T_J\m{J}$ that diagonalises~$\Xi_A$. Decomposing~$TM$ as~$T^{1,0}M\oplus T^{0,1}M$, we can write~$A=\begin{pmatrix}0 & \bar{\alpha}^\vee \\ \alpha^\vee & 0\end{pmatrix}$. Moreover, in a system of complex coordinates on~$M$,~$\alpha^\vee=g^{-1}\sigma$ for a symmetric complex matrix~$\sigma$. Since~$g$ is a Hermitian matrix, there is a unitary matrix~$S$ such that~$g^{-1}=\bar{S}^\transpose\Lambda S$, for a diagonal matrix~$\Lambda$ with real, positive eigenvalues. Then we can decompose~$\alpha^\vee$ as
\begin{equation*}
\alpha^\vee=\bar{S}^\transpose\Lambda S\sigma=\bar{S}^\transpose\Lambda^{\frac{1}{2}}\Lambda^{\frac{1}{2}} S\sigma S^\transpose\Lambda^{\frac{1}{2}}\Lambda^{-\frac{1}{2}}\bar{S}.
\end{equation*}
The matrix~$R:=\Lambda^{\frac{1}{2}} S\sigma S^\transpose\Lambda^{\frac{1}{2}}$ is symmetric, and the Takagi-Autonne factorization (see~\cite{Takagi_Factor} and~\cite{Autonne_Factor}) tells us that there is a unitary matrix~$U$ such that~$R=U D U^\transpose$ for a real diagonal matrix~$D$. Let~$\Delta:=D^2$. Then we have
\begin{align*}
&\alpha^\vee=\bar{S}^\transpose\Lambda^{\frac{1}{2}}U D U^\transpose \Lambda^{-\frac{1}{2}}\bar{S},\\
&\alpha^\vee\bar{\alpha}^\vee=\bar{S}^\transpose\Lambda^{\frac{1}{2}}U \Delta \bar{U}^\transpose \Lambda^{-\frac{1}{2}}S.
\end{align*}
This shows that the eigenvalues of~$\alpha^\vee\bar{\alpha}^\vee$ are the diagonal entries of~$\Delta$, and in particular they are all real and non-negative.

From now on let~$\Delta=\mrm{diag}\left(\delta(1),\dots,\delta(n)\right)$, and let~$Q:=\bar{S}^\transpose\Lambda^{\frac{1}{2}}U$, so that~$\alpha^\vee\bar{\alpha}^\vee=Q\Delta Q^{-1}$. The automorphism~$Q$ essentially describes a basis of~$T_J\m{J}$ for which~$\Xi_A$ is diagonal: indeed, for any~$B$ we find
\begin{equation*}
\begin{split}
A^2B+BA^2=&\begin{pmatrix}0 & \bar{Q}\\ Q & 0\end{pmatrix}
\begin{pmatrix}\Delta & 0\\ 0 &\Delta\end{pmatrix}
\begin{pmatrix}0 & Q^{-1}\\ \bar{Q}^{-1} & 0\end{pmatrix}
\begin{pmatrix}0& B'\\ B''&0\end{pmatrix}+\\
&+\begin{pmatrix}0& B'\\ B''&0\end{pmatrix}
\begin{pmatrix}0 & \bar{Q}\\ Q & 0\end{pmatrix}
\begin{pmatrix}\Delta & 0\\ 0 &\Delta\end{pmatrix}
\begin{pmatrix}0 & Q^{-1}\\ \bar{Q}^{-1} & 0\end{pmatrix}=\\
=\begin{pmatrix}0 &\bar{Q}\\ Q &0\end{pmatrix}&\!\!
\begin{pmatrix}0 &\Delta P(B)+P(B)\Delta\\ \Delta\overline{P(B)}+\overline{P(B)}\Delta &0\end{pmatrix}\!\!
\begin{pmatrix}0 &Q^{-1}\\\bar{Q}^{-1} &0\end{pmatrix},
\end{split}
\end{equation*}
where~$P(B)$ denotes the symmetric matrix~$Q^{-1}B''\bar{Q}$. The map
\begin{equation*}
P\mapsto \Delta P+P\Delta
\end{equation*}
on the space of complex symmetric matrices is quite easy to diagonalise: considering the matrices~$E_{ij}$ for~$i\leq j$ defined by
\begin{equation*}
\left(E_{ij}\right)_{pq}:=\frac{1}{2}\left(\delta_{ip}\delta_{jq}+\delta_{iq}\delta_{jp}\right),
\end{equation*}
we have~$\Delta E_{ij}+E_{ij}\Delta=\left(\delta_i+\delta_j\right)E_{ij}$.

We are now ready to compute the Biquard-Gauduchon function. We do this under the change of basis defined by~$Q$. Notice that~$P(A)=D$ is a real diagonal matrix, so that in the basis~$E_{ij}$ of the space of complex symmetric matrices~$P(A)=\sum_{i=1}^n\sqrt{\delta_i}\,E_{ii}$. Then we have
\begin{align}\label{eq:rhoAC_acc}
\nonumber&\rho(J,\alpha)=\frac{1}{2}\mrm{Tr}\left[f(\Xi_A)(A)\!\cdot\!A\right]=\\
\nonumber&=\frac{1}{2}\mrm{Tr}\left[\!\begin{pmatrix}
0 & \sum_{i=1}^nf(-\delta_i)\sqrt{\delta_i}\,E_{ii}
\\ \sum_{i=1}^nf(-\delta_i)\sqrt{\delta_i}\,E_{ii} & 0
\end{pmatrix}\!\!
\begin{pmatrix}
0 & P(A) \\ \overline{P(A)} & 0
\end{pmatrix}
\!\right]=\\
&=\mrm{Tr}\left[\left(\sum_{i=1}^nf(-\delta_i)\sqrt{\delta_i}\,E_{ii}\right)\left(\sum_{i=1}^n\sqrt{\delta_i}\,E_{ii}\right)\right].
\end{align}
Note that~$E_{ij}$ is an orthogonal basis with respect to the trace, and in particular
\begin{equation*}
\mrm{Tr}(E_{ii}E_{jj})=\sum_{k,l}\delta_{ik}\delta_{il}\delta_{jk}\delta_{jl}=\sum_l\delta_{il}\delta_{ij}\delta_{jl}=\delta_{ij}.
\end{equation*}
So~\eqref{eq:rhoAC_acc} gives 
\begin{equation*}
\begin{split}
\rho(J,\alpha)=&\mrm{Tr}\left[\left(\sum_{i=1}^nf(-\delta_i)\sqrt{\delta_i}E_{ii}\right)\left(\sum_{i=1}^n\sqrt{\delta_i}E_{ii}\right)\right]=\sum_{i=1}^nf(-\delta_i)\delta_i=\\
=&\sum_i1-\sqrt{1-\delta_i}+\mrm{log}\frac{1+\sqrt{1-\delta_i}}{2}.
\end{split}
\end{equation*}
In other words, if we let~$\tilde{f}(x):=1-\sqrt{1-x}+\mrm{log}\frac{1+\sqrt{1-x}}{2}$, then~$\rho(J,\alpha)$ is the spectral function~$\mrm{Tr}\left(\tilde{f}(\alpha\bar{\alpha})\right)$.

Now the integral in~\eqref{eq:real_mm_implicit} does not depend on~$\rho$, but rather on its differential~$\mrm{d}\rho$. This can be computed as
\begin{equation*}
\begin{split}
\diff_t\Big|_{t=0}\rho(\alpha_t\bar{\alpha}_t)=&\mrm{Tr}\left(\tilde{f}'(\alpha\bar{\alpha})\diff_t(\alpha_t\bar{\alpha}_t)\right)=\\
=&\mrm{Tr}\left(\frac{1}{2}\left(\mathbbm{1}+\left(\mathbbm{1}-\alpha\bar{\alpha}\right)^{\frac{1}{2}}\right)^{-1}\diff_t(\alpha_t\bar{\alpha}_t)\right)=\mrm{Tr}\left(\hat{\alpha}\,\diff_t(\alpha_t\bar{\alpha}_t)\right).
\end{split}
\end{equation*}
At this point the general calculations in~\cite[\S$4.2$ and \S$5.2$]{ScarpaStoppa_hyperk_reduction} can be used to obtain the required explicit expression for the real moment map.\qedhere
\end{proof}

\begin{remark}
We will also need an alternative expression for~$\hat{\alpha}$; at least when the eigenvalues of~$\alpha\bar{\alpha}$ are distinct, one can check in a system of coordinates for which~$\alpha\bar{\alpha}$ is diagonal that
\begin{equation}\label{eq:hat_alpha}
\hat{\alpha}=\sum_i\frac{1}{2\left(1+\sqrt{1-\delta_i}\right)}\prod_{j\not= i}\frac{\delta_j\mathbbm{1}-\alpha\bar{\alpha}}{\delta_j-\delta_i}.
\end{equation}
By passing to the limit as~$\delta_j - \delta_i \to 0$, this expression for~$\hat{\alpha}$ extends to the case in which two or more eigenvalues coincide. 
\end{remark}

\section{The HcscK equations in symplectic coordinates}\label{SympCoordSec}

This Section is devoted to writing the Hitchin-cscK system in a system of symplectic coordinates on abelian varieties and toric manifolds. We obtain the expression of the periodic HcscK system of Theorem~\ref{thm:HcscKSymplCoordThm} and the analogous equation~\eqref{eq:symplectic_HcscK_toric} in the toric case. We also study the boundary conditions for a symplectic potential~$u$ and a Higgs term~$\xi$ on a convex polytope corresponding to a toric manifold, proving Proposition~\ref{proposition:toric_boundarycond}. The techniques used in this proof can be used to study the integrability condition for a Higgs fields~$\xi$. This is the content of Proposition~\ref{proposition:integrability_xi}.

These results rely on the interplay of complex and symplectic coordinates. The idea of developing the K\"ahler geometry of toric manifolds in symplectic coordinates is due to Abreu~\cite{Abreu_toric} and was adapted to complex tori in~\cite{FengSzekelyhidi_abelian}. 

Consider first the case of an abelian variety~$M=\mathbb{C}^n/(\mathbb{Z}^n+i\mathbb{Z}^n)$, with the flat K\"ahler metric~$\omega$. Any~$\bb{T}^n$-invariant K\"ahler form~$\omega_g$ in the K\"ahler class~$[\omega]$ is given by the expression
\begin{equation*}
\omega_g = i \sum_{a,b} v_{a b} dz^a\wedge d\bar{z}^b
\end{equation*}
where~$v(x)$ is a strictly convex function of the form
\begin{equation*}
v(x)=\frac{1}{2}\abs{x}^2+h(x),
\end{equation*}
and the Hessian~$D^2 v = v_{ab}$ is~$\bb{Z}^n$-periodic. By convexity, the gradient
\begin{equation*}
y = \nabla v(x)
\end{equation*}
provides an alternative system of (periodic) coordinates; one can check that in fact these are Darboux coordinates for the symplectic form~$\omega_g$, so the~$y$ are called \emph{symplectic coordinates}. The Legendre dual~$u(y)$ to the potential~$v(x)$ is defined by the involutive relation
\begin{equation*}
u(y) + v(x) = x\cdot y.
\end{equation*}
Then the~$\mathbb{T}^n$-invariant tensor
\begin{equation*}
J(y, w) = J(y) = 
\left(\begin{matrix} 0 & -(D^2 u)^{-1} \\
D^2 u & 0
\end{matrix}\right) 
\end{equation*}
is an integrable almost complex structure on~$M$ compatible with the reference form~$\omega$. Considering (periodic) complex coordinates for~$J$ we return to the viewpoint of the fixed almost complex structure~$J_0$ on the abelian variety~$M$ with varying K\"ahler form~$\omega_g$. 

We have the fundamental Legendre duality identities 
\begin{equation*}
u_{ab}(y)=v^{ab}(x), v_{ab}(x)=u^{ab}(y), v^{ab}\diff_{x^a}=\diff_{y^b}, u^{ab}\diff_{y^a}=\diff_{x^b}.
\end{equation*}
Using these identities one may obtain the well-known result of Abreu~\cite{Abreu_toric} expressing the scalar curvature of the K\"ahler metric~$g_J$ as
\begin{equation*}
s(g_J) = -\frac{1}{4}(u^{ab})_{,ab}.
\end{equation*}

The key ingredient in the proof of Theorem~\ref{thm:HcscKSymplCoordThm} is the following simple computation.
\begin{lemma}\label{lemma:divergenza_coord_sympl}
Let~$X$ be a~$(1,0)$-vector field on~$M$ that is invariant under the~$\mathbb{T}^n$-action. Then in symplectic coordinates the divergence of~$X$ may be expressed as 
\begin{equation*}
\operatorname{div}(X)=\frac{1}{2}\diff_{y^a}\left(u^{ab}\chi_b(y)\right)
\end{equation*}
where we write~$\chi_b(y) =X^b(x)$ for the image under the Legendre duality.
\end{lemma}
\begin{proof}
Using the fundamental relations of Legendre duality, we have
\begin{equation*}
\begin{split}
\nabla_aX^a=&\frac{1}{2}\left(\partial_{x^a}X^a(x)+X^b(x)v^{ac}\partial_av_{bc}\right)\\
=&\frac{1}{2}\left(u^{ab}\partial_{y^b}\chi_a(y)+\chi_b(y)\partial_{y^c}u^{bc}\right)=\frac{1}{2}\partial_{y^a}\left(u^{ab}\chi_b(y)\right).
\end{split}
\end{equation*}\qedhere
\end{proof}
Let us consider the complex moment map equation
\begin{equation*} 
m^{\bb{C}}_{(J,\alpha)}(h) = 0, \, h \in \mathfrak{ham}(M, \omega).
\end{equation*}
According to~\cite{ScarpaStoppa_hyperk_reduction} this is equivalent to the linear PDE
\begin{equation*}
\operatorname{div}(\diff^* A^{0,1}) = 0
\end{equation*}
where~$A = \alpha^{\vee} = A^{1,0} + A^{0,1}$ is the type decomposition of the endomorphism~$A$ of~$T_{\bb{C}}M$ dual to the cotangent vector~$\alpha$ and~$\diff^*$ denotes the formal adjoint with respect to~$g_J=\omega(-,J-)$. Using complex coordinates we may write
\begin{equation*}
A^{0,1}=A^{\bar{b}}_c\,dz^c\otimes\diff_{\bar{z}^b}=g^{a\bar{b}}\varphi_{ac}\,dz^c\otimes\diff_{\bar{z}^b}
\end{equation*}
for some bilinear form~$\varphi$ on the fibres of~$T^{1,0}M$. The required compatibility between~$A$,~$J$ and~$\omega$ is equivalent to symmetry of this bilinear form, so we have in fact~$\varphi\in\operatorname{Sym}^2({T^{1,0}}^*M)$. This is true in general, and in our present~$\mathbb{T}^n$-invariant situation we have moreover~$A^{0,1}=v^{ac}(x)\varphi_{cb}(x)dz^b\otimes\diff_{\bar{z}^a}$. 

We are now in a position to express the complex moment map equation in symplectic coordinates. Let us write~$\xi$ for the image of the matrix-valued function~$\varphi$ under Legendre duality, i.e.~$\xi^{ab}(y) = \varphi_{ab}(x)$. Then we have~$A^{0,1}(x) = u_{ac}(y) \xi^{cb}(y)$, and 
\begin{equation*}
\begin{split}
\left(\diff^*A^{0,1}\right)^a=&-g^{\bar{c}b}\nabla_{\bar{c}}A^{\bar{a}}_b =-\frac{1}{2}v^{cb}\left(\diff_c(v^{ad}\varphi_{bd})+v^{ed}\varphi_{bd}v^{af}\diff_cv_{ef}\right)=\\
=&-\frac{1}{2}\left(\diff_{y^b}(u_{ad}\xi^{bd})+u_{ed}\xi^{bd}u_{af}\diff_{y^b}u^{ef}\right)=-\frac{1}{2}u_{ad}\diff_{y^b}\xi^{bd}
\end{split}
\end{equation*}
so by Lemma~\ref{lemma:divergenza_coord_sympl} we find
\begin{equation*}
\operatorname{div}\left(\diff^*A^{0,1}\right)=-\frac{1}{4}\diff_{y^c}\left(u^{ac}u_{ad}\diff_{y^b}\xi^{bd}\right)=-\frac{1}{4}\diff_{y^c}\diff_{y^b}\xi^{bc}.
\end{equation*}
Thus the vanishing condition~$m^{\bb{C}}_{(J,\alpha)}(h) = 0, \, h \in \mathfrak{ham}(M, \omega)$ is equivalent to the linear PDE 
%\begin{equation*}
$\xi^{ab}_{,ab}=0$
%\end{equation*}
appearing in~\eqref{eq:HcscK_squareroot}.

We turn to the real moment map equation of the HcscK system, starting from the expression of the moment map in Theorem~\ref{thm:real_mm_expl}. We need a preliminary result about the local coordinate expression of the tensor~$\hat{\alpha}$ .

\begin{lemma}\label{lemma:reale}
Let~$\check{\alpha}\indices{^a_b}(y):=\hat{\alpha}\indices{_a^b}(x)$ be the coordinate expression of~$\hat{\alpha}$ in symplectic coordinates, so that~$\check{\alpha}=\frac{1}{2}\left(\mathbbm{1}+\left(\mathbbm{1}-\xi\,D^2u\,\bar{\xi}\,D^2u\right)^{\frac{1}{2}}\right)^{-1}$. Then~$\left(\check{\alpha}\,\xi\,D^2u\,\bar{\xi}\right)^{ab}_{,ab}$ is a real quantity.
\end{lemma}
\begin{proof}
Using the formula~\eqref{eq:hat_alpha} for~$\hat{\alpha}$, we see that~$\bar{\alpha}\hat{\alpha}=\bar{\hat{\alpha}}\bar{\alpha}$ and~$\hat{\alpha}\alpha=\alpha\bar{\hat{\alpha}}$. Turning to the expression of~$\alpha$ in symplectic coordinates, this translates to~$\check{\alpha}\,\xi\,D^2u=\xi\,D^2u\,\bar{\check{\alpha}}$. Since~$\xi$ is symmetric, we also have~$\bar{\check{\alpha}}\bar{\xi}=\bar{\xi}\bar{\check{\alpha}}^\transpose$, so that
\begin{equation*}
\left(\check{\alpha}\,\xi\,D^2u\,\bar{\xi}\right)^{ab}_{,ab}=\left(\xi\,D^2u\,\bar{\xi}\,\overline{\check{\alpha}^{\transpose}}\right)^{ab}_{,ab}=\left(\left(\xi\,D^2u\,\bar{\xi}\,\overline{\check{\alpha}^{\transpose}}\right)^\transpose\right)^{ab}_{,ab}=\left(\overline{\check{\alpha}}\,\bar{\xi}\,D^2u\,\xi\right)^{ab}_{,ab}.{}
\end{equation*}\qedhere
\end{proof}
Recall from Theorem~\ref{thm:real_mm_expl} that the real moment map can be written as
\begin{equation}\label{eq:real_HcscK}
2\left(s(g_J)-\hat{s}\right)+\operatorname{div}X(J,\alpha)=0.
\end{equation}
Using Abreu's formula for the scalar curvature, together with the fact that~$\hat{s}=0$ on~$M$, we know
\begin{equation*}
2\left(s(g_J)-\hat{s}\right)=-\frac{1}{2}u^{ab}_{,ab}-2\hat{s}=-\frac{1}{2}u^{ab}_{,ab}.
\end{equation*}
In order to prove Theorem~\ref{thm:HcscKSymplCoordThm} we will compute the divergence term of~\eqref{eq:real_HcscK} in symplectic coordinates. We will use the following fact: if~$X$ is a real vector field and we decompose it as~$X=X^{1,0}+X^{0,1}$, we have~$X^{0,1}=\overline{X^{1,0}}$, so~$X=2\,\Re (X^{1,0})$ and
\begin{equation*}
\operatorname{div}(X)=\operatorname{div}\left(2\,\Re (X^{1,0})\right)=2\,\Re \left(\operatorname{div}(X^{1,0})\right).
\end{equation*}

\begin{proof}[Proof of Theorem~\ref{thm:HcscKSymplCoordThm}]
Fix a symplectic potential~$u$, and let~$G=D^2u$ be the Hessian of~$u$. The~$(1,0)$-part of the vector field~$X(J,\alpha)$ in~\eqref{eq:real_HcscK} is
\begin{equation}\label{eq:X_10}
X(J,\alpha)^a=g(\nabla^a\alpha,\bar{\alpha}\,\hat{\alpha})-g(\nabla^a\bar{\alpha},\alpha\,\bar{\hat{\alpha}})-2\nabla^*(\alpha\bar{\alpha}\,\hat{\alpha})^a.
\end{equation}
We examine the three terms in~\eqref{eq:X_10} separately:
\begin{align*}
& g(\nabla^a\alpha,\bar{\alpha}\,\hat{\alpha})=\frac{1}{2}\diff_{y^a}\xi^{bc}\,u_{bd}u_{ep}\bar{\xi}^{dp}\check{\alpha}\indices{^e_c}\\
& g(\nabla^a\bar{\alpha},\alpha\,\bar{\hat{\alpha}})=g(\nabla^a\bar{\alpha},\hat{\alpha}\alpha)=\frac{1}{2}\left(\diff_{y^a}\bar{\xi}^{bc}u_{bd}\check{\alpha}\indices{^d_e}\xi^{ep}u_{pc}+2\diff_{y^a}u_{bd}\bar{\xi}^{bc}\check{\alpha}\indices{^d_e}\xi^{ep}u_{pc}\right)\\
& \nabla^*(\alpha\bar{\alpha}\hat{\alpha})^a=-\frac{1}{2}\diff_{y^c}\left(\xi^{cp}u_{dp}\bar{\xi}^{db}u_{be}\check{\alpha}\indices{^e_a}\right).
\end{align*}
Putting everything together we have
\begin{equation}\label{eq:symplectic_acc1}
\begin{split}
2\,X^a({x})=&\diff_{a}\xi^{cb}\,u_{bd}(\bar{\xi}G)\indices{^d_e}\,\check{\alpha}\indices{^e_c}-\diff_{a}\bar{\xi}^{bc}\,u_{bd}\,\check{\alpha}\indices{^d_e}(\xi G)\indices{^e_c}-\\
&-2\diff_{y^a}u_{bd}\,\bar{\xi}^{bc}\check{\alpha}\indices{^d_e}(\xi G)\indices{^e_c}+2\diff_{y^c}\left((\xi G)\indices{^c_d}(\bar{\xi}G)\indices{^d_e}\check{\alpha}\indices{^e_a}\right).
\end{split}
\end{equation}
Since~$\check{\alpha}\,\xi\,G=\xi\,G\,\bar{\check{\alpha}}$ we can rewrite the first two terms in~\eqref{eq:symplectic_acc1} as
\begin{equation*}
\begin{gathered}
\diff_{a}\xi^{cb}\,u_{bd}(\bar{\xi}G)\indices{^d_e}\,\check{\alpha}\indices{^e_c}-\diff_{a}\bar{\xi}^{bc}\,u_{bd}\,\check{\alpha}\indices{^d_e}(\xi G)\indices{^e_c}=\\=\diff_{a}\xi^{cb}\,u_{bd}(\bar{\xi}G)\indices{^d_e}\,\check{\alpha}\indices{^e_c}-\diff_{a}\bar{\xi}^{bc}\,u_{bd}\,\bar{\check{\alpha}}\indices{^e_c}\,(\xi G)\indices{^d_e}
\end{gathered}
\end{equation*}
and in particular this quantity is purely imaginary; as we are interested in the real part of the divergence of~$X^{1,0}$, we can ignore these two terms in the computation. For the other terms in~\eqref{eq:symplectic_acc1} we have instead
\begin{align*}
&-2\diff_{y^a}u_{bd}\,(\check{\alpha}\xi G\bar{\xi})^{db}+2\diff_{y^c}(\xi G\bar{\xi}G\check{\alpha})\indices{^c_a}=\\&=-2\diff_{y^a}u_{bd}\,(\check{\alpha}\xi G\bar{\xi})^{db}+2\diff_{y^c}(\check{\alpha}\xi G\bar{\xi}G)\indices{^c_a}=\\
&=-2\diff_{y^a}u_{bd}\,(\check{\alpha}\xi G\bar{\xi})^{db}
+2\,u_{ea}\,\diff_{y^c}(\check{\alpha}\xi G\bar{\xi})^{ce}
+2\,\diff_{y^c}u_{ea}\,(\check{\alpha}\xi G\bar{\xi})^{ce}=\\
&=2\,u_{ea}\,\diff_{y^c}(\check{\alpha}\xi G\bar{\xi})^{ce}
\end{align*}
and finally we conclude the computation of the divergence term using Lemma~\ref{lemma:divergenza_coord_sympl} and Lemma~\ref{lemma:reale}:
\begin{equation*}
\mrm{div}(X(J,\alpha))=\Re\left(\diff_{y^a}\diff_{y^b}\left(\check{\alpha}\,\xi\,G\,\bar{\xi}\right)^{ab}\right)=\diff_{y^a}\diff_{y^b}\left(\check{\alpha}\,\xi\,G\,\bar{\xi}\right)^{ab}.
\end{equation*}
Combining Abreu's formula for the scalar curvature and this expression for the divergence of~$X(J,\alpha)$, we can write the real moment map equation as
\begin{equation}\label{eq:HcscKSymplRealEq}
\left(G^{-1}-\left(\mathbbm{1}+\left(\mathbbm{1}-\xi\,G\,\bar{\xi}\,G\right)^{\frac{1}{2}}\right)^{-1}\!\!\!\xi\,G\,\bar{\xi}\right)^{ab}_{,ab}=0. 
\end{equation}
To conclude the proof, notice that~$x\left(1+\sqrt{1-x}\right)^{-1}=1-\sqrt{1-x}$, so that in a basis that diagonalises~$\xi\,G\,\bar{\xi}\,G$ we get
\begin{equation*}
\mathbbm{1}-\left(\mathbbm{1}+\left(\mathbbm{1}-\xi\,G\,\bar{\xi}\,G\right)^\frac{1}{2}\right)^{-1}\!\!\!\xi\,G\,\bar{\xi}\,G=\left(\mathbbm{1}-\xi\,G\,\bar{\xi}\,G\right)^{\frac{1}{2}}
\end{equation*}
and~\eqref{eq:HcscKSymplRealEq} then becomes~$\left(\left(\mathbbm{1}-\xi\,G\,\bar{\xi}\,G\right)^{\frac{1}{2}}G^{-1}\right)^{ab}_{,ab}=0$.\qedhere
\end{proof}

The same proof also works in the toric case, as it only uses the Legendre duality between K\"ahler potentials and symplectic potentials. Thus the HcscK equations in the interior of a Delzant polytope~$P$ are given by 
\begin{equation*}
\begin{dcases}
\xi^{ab}_{,ab}=0\\
\left(\left(\mathbbm{1}-\xi\,D^2u\,\bar{\xi}\,D^2u\right)^{\frac{1}{2}}D^2u^{-1}\right)^{ab}_{,ab}=-C
\end{dcases}
\end{equation*}
for~$C=4\hat{s}$. We need to impose some boundary conditions on the symplectic potential~$u$ and the Higgs term~$\xi$, to guarantee that a solution of the HcscK system on~$P$ can be extended to a solution on the whole toric manifold~$M$. We start by recalling the boundary conditions for the symplectic potentials, following Abreu~\cite{Abreu_toric}.

There is a standard complex structure~$J_P$ on~$P^\circ\times\bb{T}^n$, whose boundary behaviour allows it to be extended to a complex structure on the whole manifold~$M_P$; moreover,~$(M_P,J_P)$ is~$\bb{T}^n$-equivariantly biholomorphic to~$(M,J)$. We recall from~\cite{Guillemin_toric} the construction of~$J_P$: let~$S_1,\dots,S_r$ be the faces of the polytope~$P$, defined as~$S_r:=\set*{\ell^r({y})=0}$ for some affine-linear functions~$\ell^r$ such that
\begin{equation}\label{eq:polytope_eq}
P=\set*{{y}\in\bb{R}^n\tc\ell^r({y})\geq 0\ \forall r}.
\end{equation}
Then the potential~$u_P$ for the canonical complex structure~$J_P$ is
\begin{equation}\label{eq:can_potential}
u_P:=\sum_r\ell^r({y})\,\mrm{log}\,\ell^r({y}).
\end{equation}
The following result of Abreu describes all possible integrable compatible complex structures~$J$ on~$P$; more precisely, it shows that any~$\bb{T}^n$-invariant complex structure on~$M$ has the same behaviour of~$J_P$ near the boundary of the moment polytope: this boundary behaviour is what we refer as ``Guillemin's boundary conditions'' for a symplectic potential on the polytope~$P$.
\begin{theorem}[\cite{Abreu_toric}, Thm.~$2.8$]\label{thm:complex_str}
Every integrable compatible complex structure~$J$ is given by a potential~$u=u_P+h$, where~$u_P({y})$ is defined by~\eqref{eq:can_potential}, and~$h$ is a smooth function on the whole polytope such that~$\mrm{Hess}_{{y}}u$ is positive definite on~$P^\circ$ and has determinant of the form
\begin{equation*}
\mrm{det}(\mrm{Hess}_{{y}}(u))=\left(\delta({y})\prod_{r}\ell^r({y})\right)^{-1}
\end{equation*}
for some strictly positive function~$\delta\in\m{C}^0(P)$.
\end{theorem}

Using Theorem~\ref{thm:complex_str} we can describe all the~$\bb{T}^n$-invariant deformations of a complex structure~$J$ in symplectic coordinates, proving Proposition~\ref{proposition:toric_boundarycond}.

Assume that~$J$ is a complex structure on the toric manifold~$(M,\omega)$. The tangent space of~$\m{J}$ at~$J$ is given by tensors~$\dot{J}\in\Gamma(M,\mrm{End}(TM))$ satisfying the relations
\begin{equation}\label{eq:defJ_relations}
J\dot{J}+\dot{J}J=0\mbox{ and }\omega(J-,\dot{J}-)+\omega(\dot{J}-,J-)=0.
\end{equation}
In the symplectic coordinates on~$P^\circ$, the complex structure is given by
\begin{equation*}
J=\begin{pmatrix}0 &-D^2u^{-1}\\ D^2u & 0\end{pmatrix}
\end{equation*}
for some symplectic potential~$u$. The conditions~\eqref{eq:defJ_relations} imply that any deformation~$\dot{J}$ of the complex structure is given on~$P^\circ$ by a matrix-valued function  
\begin{equation}\label{eq:deformation_dotJ}
\dot{J}=\begin{pmatrix} D^2u^{-1}A & D^2u^{-1}\,B\,D^2u^{-1}\\ B &- A\,D^2u^{-1}\end{pmatrix}
\end{equation}
for~$A$,~$B$ symmetric matrices. In particular 
\begin{equation*}
\dot{J}=\begin{pmatrix} 0 & D^2u^{-1}\,D^2\varphi\,D^2u^{-1}\\ D^2\varphi & 0\end{pmatrix}.
\end{equation*}
corresponds to the path of symplectic potentials~$u+\varepsilon\varphi$ for~$\varphi\in\m{C}^\infty(P)$. Note that~$J\dot{J}$ is also an integrable first-order deformation of~$J$, so any function given by~\eqref{eq:deformation_dotJ} for two Hessian matrices~$A,B\in\m{C}^\infty(P,\bb{R}^{n\times n})$ corresponds to an integrable first-order deformation of~$J$. Theorem~\ref{thm:complex_str} implies that any integrable deformation of~$J$ can be written in this way. If we do not assume~$A$ and~$B$ to be Hessians, but rather just symmetric matrices, we obtain deformations of the complex structure which are not integrable. This description of first-order deformations of the complex structure~\eqref{eq:deformation_dotJ} allows us to find the correct boundary behaviour for a Higgs term~$\xi$.

\begin{proof}[Proof of Proposition~\ref{proposition:toric_boundarycond}]
Let~$J$ be a complex structure on the toric manifold~$M$, corresponding to a symplectic potential~$u$ on the Delzant polytope~$P$. Consider also a first-order deformation~$\dot{J}$ of~$J$, defined by two symmetric matrix-valued functions~$A,B\in\m{C}^\infty(P,\bb{R})$ as in~\eqref{eq:deformation_dotJ}. To~$J$ and~$\dot{J}$ corresponds a Higgs matrix~$\xi$; the relation between~$\xi$,~$u$ and~$\dot{J}$ is given by 
\begin{equation*}
\xi^{ac}(y)=\dot{J}\indices{^{\bar{b}}_a}(x)\,g_{c\bar{b}}(x),
\end{equation*}
where~$\dot{J}\indices{^{\bar{b}}_a}(x)\mrm{d}z^a\otimes\diff_{\bar{z}^b}$ is the~$(0,1)$-part of~$\dot{J}$. If we let~$G:=D^2u$, the system of holomorphic coordinates for~$J$ is described by the vector fields~$\diff_{z^a}=\frac{1}{2}\left(G^{ab}\diff_{y^b}-\I\diff_{w^a}\right)$, so the~$(0,1)$-part of~$\dot{J}$ is
\begin{equation*}
\begin{split}
\dot{J}\indices{^{\bar{b}}_a}=&\mrm{d}\bar{z}^b\left(\dot{J}(\diff_{z^a})\right)=\frac{1}{2}\left(G_{bc}\mrm{d}y^c-\I\mrm{d}w^b\right)\left(G^{ad}\dot{J}(\diff_{y^d})-\I\dot{J}(\diff_{w^a})\right)=\\
=&\frac{1}{2}\left(G_{bc}G^{ad}G^{ef}A_{fd}\delta^c_e+\delta^{ef}A_{fp}G^{pq}\delta_{qa}\delta^b_e
\right)-\\
&-\frac{\I}{2}\left(G_{bc}G^{ef}B_{fp}G^{pq}\delta_{aq}\delta^c_e+G^{ad}\delta^{ef}B_{fd}\delta^b_e\right)=G^{ad}A_{db}-\I G^{ad}B_{db}.
\end{split}
\end{equation*}
For the matrix~$\xi$ then we get
\begin{equation}\label{eq:def_cs_sympcoord}
\xi^{ac}=\left(G^{ad}A_{db}-\I G^{ad}B_{db}\right)G^{cb}=\left(G^{-1}\left(A-\I B\right)G^{-1}\right)^{ac}.
\end{equation}
We can rephrase this as follows: let~$u$ be a symplectic potential, defining a complex structure~$J$ on the toric manifold~$M$. A symmetric matrix-valued function~$\xi$ comes from a~$\bb{T}^n$-invariant Higgs term~$\alpha\in T^*_J\!\m{J}$ if and only if~$D^2u\,\xi\,D^2u=A-\I\,B$ is smooth on~$P$.

So we see that the boundary behaviour of~$\xi$ is determined by that of~$u$. Indeed since~$u$ satisfies Guillemin's boundary conditions (see Theorem~\ref{thm:complex_str}),~$u$ is written as~$u_P+h$ for some~$h\in\m{C}^\infty(P)$. Then
\begin{align*}
&D^2u\,\xi\,D^2u=\left(D^2u_P+D^2h\right)\xi\left(D^2u_P+D^2h\right)=\\
&=\left(\mathbbm{1}+D^2h\,D^2u_P^{-1}\right)D^2u_P\,\xi\,D^2u_P\left(\mathbbm{1}+D^2u_P^{-1}D^2h\right)
\end{align*}
and~$D^2u_P^{-1}$ is smooth on~$P$, so~$D^2u\,\xi\,D^2u$ is smooth on the whole~$P$ if and only if~$D^2u_P\,\xi\,D^2u_P$ is.\qedhere
\end{proof}

We now digress for a moment in order to study the integrability condition of a Higgs term~$\alpha$ from the ``symplectic'' point of view, taking cue from the description of~$\xi$ in equation~\eqref{eq:def_cs_sympcoord}. Notice first that~\eqref{eq:def_cs_sympcoord} holds in both the toric and the periodic settings, the only difference is that in the abelian case we should assume periodicity of~$D^2u$,~$A$ and~$B$. 

The integrability condition for the deformation of the complex structure~$J$ induced by a Higgs term~$\alpha$, is the first-order Maurer-Cartan equation
\begin{equation}\label{eq:integrability}
\diff\alpha=0.
\end{equation}
The next result expresses~\eqref{eq:integrability} in terms of our symplectic coordinates.
\begin{proposition}\label{proposition:integrability_xi}
Let~$u$ be a symplectic potential, associated to an integrable complex structure~$J$, and let~$\xi$ be a symmetric matrix corresponding to a Higgs term~$\alpha\in T^*_J\m{J}$. Then~$\xi$ comes from an integrable deformation of the complex structure~$J$ if and only if~$D^2u\,\xi\,D^2u$ is the Hessian matrix of a function.
\end{proposition}
\begin{proof}
The integrability condition~\eqref{eq:integrability} is written more explicitly as
\begin{equation*}
\diff_{z^a}\alpha\indices{_c^{\bar{b}}}=\diff_{z^c}\alpha\indices{_a^{\bar{b}}}
\end{equation*}
and in our case it becomes
\begin{equation*}
\diff_{x^a}\left(v^{bd}q_{cd}\right)=\diff_{x^c}\left(v^{bd}q_{ad}\right)
\end{equation*}
i.e.~$\diff_{x^a}\left(v^{bd}q_{cd}\right)$ should be symmetric in the indices~$a$,~$c$. Under Legendre duality this quantity becomes, as~$u_{ij}$ is a Hessian matrix,
\begin{equation*}
\diff_{x^a}\left(v^{bd}q_{cd}\right)=u^{ae}\diff_{y^e}\left(u_{bd}\xi^{cd}\right)=u^{ae}\diff_{y^e}\left(u_{bd}\xi^{fd}u_{fg}\right)u^{gc}+u_{bd}\xi^{df}\diff_{y^f}u^{ac}.
\end{equation*}
The second expression on the right hand side is symmetric in~$a$ and~$c$, while the first
\begin{equation*}
u^{ae}\diff_{y^e}\left(u_{bd}\xi^{fd}u_{fg}\right)u^{gc}
\end{equation*}
is symmetric in~$a$ and~$c$ if and only if~$\diff_{y^e}\left(u_{bd}\xi^{fd}u_{fg}\right)$ is symmetric in~$e$ and~$g$. Thus letting~$H_{ab}(y):=u_{ac}\xi^{cd}u_{db}$, equation~\eqref{eq:integrability} is equivalent to~$\diff_c H_{ab}$ being symmetric in all indices.\qedhere
\end{proof}
 
\section{The periodic \textit{HK}-energy}\label{ConvexitySec}

In this Section we study the Biquard-Gauduchon functional~$\m{H}(u, \xi)$ introduced in~\eqref{eq:BGFunc}, and the corresponding version of the~$K$-energy~$\widehat{\m{F}}(u, \xi)$ defined in~\eqref{eq:HKEn}, proving Proposition~\ref{proposition:HKEulLagr} and Theorem~\ref{thm:convexity}.

\begin{proof}[Proof of Proposition~\ref{proposition:HKEulLagr}]
Let~$u(t) =u+t\,\varphi$ be a path of symplectic potentials in~$\mathcal{A}(\xi)$, for a periodic function~$\varphi$, and let~$G_t:=D^2u_t$. We can compute the derivative of the Biquard-Gauduchon function along the path~$\xi\,G_t\,\bar{\xi}\,G_t$ as in the proof of Theorem~\ref{thm:real_mm_expl}:
\begin{equation*}
\diff_t\rho(u_t,\xi)=\mrm{Tr}\left(\check{\alpha}\,\diff_t(\xi\,G_t\,\bar{\xi}\,G_t)\right)
\end{equation*}
From the proof of Lemma~\ref{lemma:reale} we have~$\hat{\alpha}\,\alpha=\alpha\,\bar{\hat{\alpha}}$ and~$\bar{\alpha}\,\hat{\alpha}=\bar{\hat{\alpha}}\,\bar{\alpha}$, so we can write the differential as
\begin{align*}
&\diff_t\rho(u_t,\xi)=\mrm{Tr}\left(\check{\alpha}\,\xi\,\dot{G}\,\bar{\xi}\,G\right)+\mrm{Tr}\left(\check{\alpha}\,\xi\,G\,\bar{\xi}\,\dot{G}\right)=\\
&=\mrm{Tr}\left(\bar{\check{\alpha}}\,\bar{\xi}\,G\,\xi\,\dot{G}\right)+\mrm{Tr}\left(\check{\alpha}\,\xi\,G\,\bar{\xi}\,\dot{G}\right)=2\,\Re\,\mrm{Tr}\left(\check{\alpha}\,\xi\,G\,\bar{\xi}\,\dot{G}\right).
\end{align*}
Computing the derivative of the Biquard-Gauduchon functional~\eqref{eq:BGFunc} is straightforward, since every~$\varphi\in T_u\m{A}$ is a periodic function:
\begin{equation*}
D\m{H}_{(u,\xi)}(\varphi)=\int\Re\left(\check{\alpha}\indices{^a_b}\xi^{bc}u_{cd}\bar{\xi}^{de}\varphi_{,ea}\right)\mrm{d}\mu=\int\Re\left(\check{\alpha}\,\xi\,G\,\bar{\xi}\right)^{ab}_{,ab}\varphi\,\mrm{d}\mu.
\end{equation*}
On the other hand the first variation of the (periodic)~$K$-energy gives the scalar curvature term, see~\cite{FengSzekelyhidi_abelian}, so the variation of~$\hat{\m{F}}=\m{F}+\m{H}$ is
\begin{equation*}
D\hat{\m{F}}_{(u,\xi)}(\varphi)=-\frac{1}{2}\int u^{ab}_{,ab}\varphi\,\mrm{d}\mu+\int\Re\left(\check{\alpha}\,\xi\,G\,\bar{\xi}\right)^{ab}_{,ab}\varphi\,\mrm{d}\mu.
\end{equation*}
To conclude the proof, just recall from Lemma~\ref{lemma:reale} that~$\left(\check{\alpha}\,\xi\,G\,\bar{\xi}\right)^{ab}_{,ab}$ is a real quantity.\qedhere
\end{proof}

We proceed to prove our convexity result, Theorem~\ref{thm:convexity}. We fix a linear path~$u_t = (1-t)u_0 + t u_1$, for~$t\in [0,1]$, between the symplectic potentials~$u_0$ and~$u_1$, and assume that the path lies in~$\m{A}(\xi)$, so that the \textit{HK}-energy is well-defined along~$u_t$. The proof of Theorem~\ref{thm:convexity} will be broken down in many steps. First, notice that~$\rho$ is a convex function of the eigenvalues of~$\xi\,D^2u\,\bar{\xi}\,D^2u$. Then, it will be convenient to consider~$\rho$ as
\begin{equation*}
\rho(u,\xi)=\rho\circ\delta(\xi\,D^2u\,\bar{\xi}\,D^2u)
\end{equation*}
where~$\delta$ is the vector of eigenvalues~$\left(\delta_1,\dots,\delta_n\right)$. We will assume for simplicity that these eigenvalues are all distinct along~$u_t$. The general case can be deduced from this special situation, as the Hessian of~$\widehat{\m{F}}$ is continuous in~$\m{A}(\xi)$.

Let~$G_t=D^2u_t$. The second variation of~$\widehat{\m{F}}$ along~$u_t$ is given by
\begin{align}\label{eq:HK_secondvar}
\nonumber &\diff^2_t\hat{\m{F}}(u_t)=\\
\nonumber&=\frac{1}{2}\int_X\mrm{Tr}\left(G_t^{-1}\dot{G}G_t^{-1}\dot{G}\right)\mrm{d}\mu+\frac{1}{2}\int_XD\rho\cdot\diff^2_t\delta(\xi\,G_t\,\bar{\xi}\,G_t)\,\mrm{d}\mu+\\
&+\frac{1}{2}\int_X\diff_t\delta(\xi\,G_t\,\bar{\xi}\,G_t)^\transpose\cdot D^2\rho\cdot \diff_t\delta(\xi\,G_t\,\bar{\xi}\,G_t)\,\mrm{d}\mu.
\end{align}
The third integral is positive by the convexity of~$\rho$. We will show that, while the second term might be negative, the sum of the first and second terms in~\eqref{eq:HK_secondvar} is always non-negative. The first step is to compute~$\diff^2_t\delta(\xi\,G_t\,\bar{\xi}\,G_t)$.

\subsection{Second variation of the eigenvalues}

Let~$\alpha_t=\xi\,G_t$, so that~$\xi\,G_t\,\bar{\xi}\,G_t=\alpha_t\bar{\alpha}_t$. An expression for the second derivative of~$\delta_i$ can be found in~\cite{Magnus_differential_eigenvalues} (see also~\cite[ch.$8$, \S$11$]{Magnus_matrix_diffcalc} for the case of a real, symmetric matrix): the computations in~\cite{Magnus_differential_eigenvalues} are carried out along a linear path, so they need to be adapted to our case, for which~$\alpha_t\bar{\alpha}_t$ is quadratic in~$t$. The result is expressed in terms of the matrix~$P_i:=\prod_{j\not= i}\frac{\delta_j-\alpha_t\bar{\alpha}_t}{\delta_j-\delta_i}$ as
\begin{equation}\label{eq:second_var_eigenval}
\begin{split}
\diff^2_t\delta_i=&\hphantom{+}\mrm{Tr}\left[P_i\,\diff^2_t(\alpha_t\bar{\alpha}_t)\right]+\\
&{+}\mrm{Tr}\left[2\,P_i\,\diff_t(\alpha_t\bar{\alpha}_t)\,(\mathbbm{1}-P_i)\left(\delta_i\mathbbm{1}-\alpha_t\bar{\alpha}_t\right)^+(\mathbbm{1}-P_i)\diff_t(\alpha_t\bar{\alpha}_t)\right],
\end{split}
\end{equation}
where we denote by~$A^+$ the \emph{Moore-Penrose inverse} of a matrix~$A$, depending on the choice of a Hermitian inner product. We refer to~\cite{Magnus_matrix_diffcalc} for background on the Moore-Penrose inverse of an endomorphism. The main property we use here is~$A^+=(A^*A)^+A^*$. Assuming that the eigenvalues of~$A$ are all distinct, in a basis for which~$A^*A$ is diagonal with eigenvalues~$\beta_j$, it can be checked that
\begin{equation*}
(A^*A)^+=\sum_{\underset{\beta_i\not=0}{i=1}}^n\frac{1}{\beta_i}\prod_{j\not=i}\frac{\beta_j\mathbbm{1}-A^*A}{\beta_j-\beta_i}.
\end{equation*}
As a consequence, we get a relatively easy expression for the Moore-Penrose inverse of~$\delta_i\mathbbm{1}-\alpha\bar{\alpha}$:
\begin{align*}
&\left(\delta_i\mathbbm{1}-\alpha\bar{\alpha}\right)^+=\\
&=\left(\sum_{b\not= i}\frac{1}{(\delta_i-\delta_b)^2}\prod_{l\not=b}\frac{(\delta_i-\delta_l)^2\mathbbm{1}-\left(\delta_i\mathbbm{1}-\alpha\bar{\alpha}\right)^*\left(\delta_i\mathbbm{1}-\alpha\bar{\alpha}\right)}{(\delta_i-\delta_l)^2-(\delta_i-\delta_b)^2}\right)\left(\delta_i\mathbbm{1}-\alpha\bar{\alpha}\right)^*,
\end{align*}
depending on the choice of a Hermitian pairing on~${T^{1,0}}^*M$. In our case we know that~$\alpha_t\bar{\alpha}_t$ is self-adjoint with respect to the metric defined by~$u_t$, so we can simplify this expression as
\begin{equation}\label{eq:pseudoinv_alpha}
\left(\delta_i\mathbbm{1}-\alpha\bar{\alpha}\right)^+=\left(\sum_{b\not= i}\frac{1}{(\delta_i-\delta_b)^2}\prod_{l\not=b}\frac{(\delta_i-\delta_l)^2\mathbbm{1}-\left(\delta_i\mathbbm{1}-\alpha\bar{\alpha}\right)^2}{(\delta_i-\delta_l)^2-(\delta_i-\delta_b)^2}\right)\left(\delta_i\mathbbm{1}-\alpha\bar{\alpha}\right).
\end{equation}
We can use equations~\eqref{eq:second_var_eigenval},~\eqref{eq:pseudoinv_alpha} to compute the second derivative of any eigenvalue~$\delta_i$ along the path~$u_t$. This in turn allows us to compute the second integrand in~\eqref{eq:HK_secondvar}:
\begin{equation*}
\begin{split}
&D\rho\cdot\diff^2_t\delta(\xi\,G_t\,\bar{\xi}\,G_t)=\sum_{i=1}^n\frac{1}{2}\frac{\diff^2_t\delta_i(\xi\,G_t\,\bar{\xi}\,G_t)}{1+\sqrt{1-\delta_i(\xi\,G_t\,\bar{\xi}\,G_t)}}=\\
=&\mrm{Tr}\left[\sum_i\frac{P_i}{1+\sqrt{1-\delta_i}}\,\xi\,\dot{G}\,\bar{\xi}\,\dot{G}\right]+\\
+&\mrm{Tr}\left[\sum_i\frac{P_i\,\diff_t(\xi\,G_t\,\bar{\xi}\,G_t)\,(\mathbbm{1}-P_i)\left(\delta_i\mathbbm{1}-\xi\,G_t\,\bar{\xi}\,G_t\right)^+(\mathbbm{1}-P_i)\diff_t(\xi\,G_t\,\bar{\xi}\,G_t)}{1+\sqrt{1-\delta_i}}\right]
\end{split}
\end{equation*}
and the first term in this sum can be written as
\begin{equation*}
\mrm{Tr}\left[\sum_i\frac{P_i}{1+\sqrt{1-\delta_i}}\,\xi\,\dot{G}\,\bar{\xi}\,\dot{G}\right]=\mrm{Tr}\left[\left(\mathbbm{1}+\left(\mathbbm{1}-\xi\,G\,\bar{\xi}\,G\right)^{\frac{1}{2}}\right)^{-1}\xi\,\dot{G}\,\bar{\xi}\,\dot{G}\right].
\end{equation*}
So we arrive at the expression
\begin{align}\label{eq:diff_seconda_autoval}
\nonumber &D\rho\cdot\diff^2_t\delta(\xi\,G_t\,\bar{\xi}\,G_t)=\mrm{Tr}\left[\left(\mathbbm{1}+\left(\mathbbm{1}-\xi\,G\,\bar{\xi}\,G\right)^{\frac{1}{2}}\right)^{-1}\xi\,\dot{G}\,\bar{\xi}\,\dot{G}\right]+\\
&+\mrm{Tr}\left[\sum_i\frac{P_i\,\diff_t(\xi\,G_t\,\bar{\xi}\,G_t)\,(\mathbbm{1}-P_i)\left(\delta_i\mathbbm{1}-\xi\,G_t\,\bar{\xi}\,G_t\right)^+(\mathbbm{1}-P_i)\diff_t(\xi\,G_t\,\bar{\xi}\,G_t)}{1+\sqrt{1-\delta_i}}\right].
\end{align}

\subsection{Preliminary computations}
In order to investigate the positivity of the quantity 
\begin{equation*}
\mrm{Tr}(G^{-1}\dot{G}\,G^{-1}\dot{G}) + D\rho\cdot\diff^2_t\delta
\end{equation*}
we will work in a system of coordinates for which~$\xi\,G\,\bar{\xi}\,G$ is diagonal, similarly to the proof of Theorem~\ref{thm:real_mm_expl}. We use the same notation: there is an orthogonal matrix~$S$, a positive diagonal matrix~$\Lambda$, a unitary matrix~$U$ and a positive diagonal matrix~$\Delta$ such that
\begin{align*}
& G=S^\transpose\Lambda\,S\\
&\xi=S^\transpose\Lambda^{-\frac{1}{2}}U \sqrt{\Delta}\, U^\transpose\Lambda^{-\frac{1}{2}}S
\end{align*}
and letting~$Q=S^\transpose\Lambda^{-\frac{1}{2}}U$ we get~$\xi\,G\,\bar{\xi}\,G=Q\Delta Q^{-1}$, so that~$\Delta=\mrm{diag}(\delta_1,\dots,\delta_n)$. If now we let~$R$ be the Hermitian matrix~$R:=Q^\transpose\dot{G}\bar{Q}$, we have the following identities:
\begin{align*}
&\mrm{Tr}(G^{-1}\dot{G}G^{-1}\dot{G})=\mrm{Tr}(R^2)=\sum_{i,j}\abs*{R\indices{^i_j}}^2,\\
&\mrm{Tr}\left[\left(\mathbbm{1}+\left(\mathbbm{1}-\xi\,G\,\bar{\xi}\,G\right)^{\frac{1}{2}}\right)^{-1}\xi\,\dot{G}\,\bar{\xi}\,\dot{G}\right]=\\
&=\mrm{Tr}\left[\left(\mathbbm{1}+\sqrt{\mathbbm{1}-\Delta}\right)^{-1}\sqrt{\Delta}\,R\,\sqrt{\Delta}\,\bar{R}\right]=\sum_{i,j}\frac{\sqrt{\delta_i\delta_j}}{1+\sqrt{1-\delta_i}}\left(R\indices{^i_j}\right)^2.
\end{align*}
We should also rewrite the second term of~\eqref{eq:diff_seconda_autoval}, the one involving the Moore-Penrose inverse of~$\delta-\xi\,G\,\bar{\xi}\,G$. Notice first that
\begin{equation}\label{eq:P_i_diag}
\mathbbm{1}-P_i=Q\cdot\mrm{diag}\left(1,\dots,1,\underset{\mbox{\textit{i}-th place}}{0},1,\dots,1\right)\cdot Q^{-1}
\end{equation}
while from~\eqref{eq:pseudoinv_alpha} we have
\begin{equation}\label{eq:pseudoinverse_diag}
\begin{gathered}
\left(\delta_i\mathbbm{1}-\xi\,G\,\bar{\xi}\,G\right)^+=\\
=Q\cdot\mrm{diag}\left(\frac{1}{\delta_i-\delta_1},\dots,\frac{1}{\delta_i-\delta_{i-1}},0,\frac{1}{\delta_i-\delta_{i+1}},\dots,\frac{1}{\delta_i-\delta_n}\right)\cdot Q^{-1}
\end{gathered}
\end{equation}
and for the derivative of~$\xi\,G_t\,\bar{\xi}\,G_t$ we find
\begin{align}\label{eq:derivative_diag}
\nonumber&\diff_t\left(\xi\,G_t\,\bar{\xi}\,G_t\right)=\xi\,\dot{G}\,\bar{\xi}\,G+\xi\,G\,\bar{\xi}\,\dot{G}=\\
&=Q\left(\sqrt{\Delta}\,R\sqrt{\Delta}+\Delta\,\bar{R}\right)Q^{-1}.
\end{align}
Combining~\eqref{eq:P_i_diag},~\eqref{eq:pseudoinverse_diag} and~\eqref{eq:derivative_diag}, we can compute the second term in~\eqref{eq:diff_seconda_autoval} as 
\begin{align}\label{eq:traccia_derivata_diag}
\nonumber&\mrm{Tr}\left[\sum_i\frac{P_i\,\diff_t(\xi\,G_t\,\bar{\xi}\,G_t)\,(\mathbbm{1}-P_i)\left(\delta_i\mathbbm{1}-\xi\,G_t\,\bar{\xi}\,G_t\right)^+(\mathbbm{1}-P_i)\diff_t(\xi\,G_t\,\bar{\xi}\,G_t)}{1+\sqrt{1-\delta_i}}\right]=\\
&=\mrm{Tr}\left[
\sum_i\mrm{diag}\left(0,\dots,0,\frac{1}{1+\sqrt{1-\delta_i}},0,\dots,0\right)\cdot
\left(\sqrt{\Delta}\,R\sqrt{\Delta}+\Delta\,\bar{R}\right)\cdot\right.\\
&\left.\cdot\mrm{diag}\left(\frac{1}{\delta_i-\delta_1},\dots,\frac{1}{\delta_i-\delta_{i-1}},0,\frac{1}{\delta_i-\delta_{i+1}},\dots,\frac{1}{\delta_i-\delta_n}\right)\cdot\left(\sqrt{\Delta}\,R\sqrt{\Delta}+\Delta\,\bar{R}\right)\right].
\end{align}
So if we let for the moment~$K:=\sqrt{\Delta}\,R\sqrt{\Delta}+\Delta\,\bar{R}$, the identity~\eqref{eq:traccia_derivata_diag} becomes 
\begin{align}\label{eq:traccia_derivata}
\nonumber&\mrm{Tr}\left[\sum_i\frac{P_i\,\diff_t(\xi\,G_t\,\bar{\xi}\,G_t)\,(\mathbbm{1}-P_i)\left(\delta_i\mathbbm{1}-\xi\,G_t\,\bar{\xi}\,G_t\right)^+(\mathbbm{1}-P_i)\diff_t(\xi\,G_t\,\bar{\xi}\,G_t)}{1+\sqrt{1-\delta_i}}\right]=\\
&=\sum_i\sum_{j\not=i}\frac{(\delta_i-\delta_j)^{-1}}{1+\sqrt{1-\delta_i}}K\indices{^i_j}K\indices{^j_i}.
\end{align}
Now it is straightforward to check that~$K\indices{^i_j}=\sqrt{\delta_i\delta_j}R\indices{^i_j}+\delta_i\bar{R}\indices{^i_j}$. 

The upshot of these computations is that the quantity 
\begin{equation*}
\mrm{Tr}(G^{-1}\dot{G}\,G^{-1}\dot{G}) + D\rho\cdot\diff^2_t\delta
\end{equation*}
can be written as the sum of three terms:
\begin{equation}\label{eq:tr_G}
\mrm{Tr}(G^{-1}\dot{G}G^{-1}\dot{G})=\sum_{i,j}\abs*{R\indices{^i_j}}^2;
\end{equation}
\begin{equation}\label{eq:tr_xi}
\mrm{Tr}\left[\left(\mathbbm{1}+\left(\mathbbm{1}-\xi\,G\,\bar{\xi}\,G\right)^{\frac{1}{2}}\right)^{-1}\xi\,\dot{G}\,\bar{\xi}\,\dot{G}\right]=\sum_{i,j}\frac{\sqrt{\delta_i\delta_j}}{1+\sqrt{1-\delta_i}}\left(R\indices{^i_j}\right)^2;
\end{equation}
\begin{align}\label{eq:tr_mista}
\nonumber&\sum_i\sum_{j\not=i}\frac{(\delta_i-\delta_j)^{-1}}{1+\sqrt{1-\delta_i}}K\indices{^i_j}K\indices{^j_i}=\\
&=\sum_i\sum_{j\not=i}\frac{(\delta_i-\delta_j)^{-1}}{1+\sqrt{1-\delta_i}}\left(2\delta_i\delta_j\abs*{R\indices{^i_j}}^2+\sqrt{\delta_i\delta_j}\left(\delta_i\left(R\indices{^i_j}\right)^2+\delta_j\left(\bar{R}\indices{^i_j}\right)^2\right)\right).
\end{align}

\subsection{Convexity of~$\widehat{\m{F}}$}
To conclude the proof of Theorem~\ref{thm:convexity}, we will show that the sum of the three terms~\eqref{eq:tr_G}, \eqref{eq:tr_xi} and~\eqref{eq:tr_mista} is non-negative, and vanishes only if~$\dot{G}=0$. It is convenient to decompose~$R$ in its real and imaginary parts, given by a symmetric matrix~$A$ and a skew-symmetric matrix~$B$ respectively. Then~\eqref{eq:tr_G} is
\begin{equation*}
\sum_{i,j}\abs*{R\indices{^i_j}}^2=\sum_{i,j}(A\indices{^i_j})^2+(B\indices{^i_j})^2
\end{equation*}
while~\eqref{eq:tr_xi} is written as
\begin{equation*}
\sum_{i,j}\frac{\sqrt{\delta_i\delta_j}}{1+\sqrt{1-\delta_i}}\left((A\indices{^i_j})^2-(B\indices{^i_j})^2+2\I A\indices{^i_j}B\indices{^i_j}\right)
\end{equation*}
and~\eqref{eq:tr_mista} becomes
\begin{align*}
&\sum_i\sum_{j\not=i}\frac{(\delta_i-\delta_j)^{-1}}{1+\sqrt{1-\delta_i}}2\delta_i\delta_j\left((A\indices{^i_j})^2+(B\indices{^i_j})^2\right)+\\
&+\sum_i\sum_{j\not=i}\frac{\sqrt{\delta_i\delta_j}}{1+\sqrt{1-\delta_i}}\frac{\delta_i+\delta_j}{\delta_i-\delta_j}\left((A\indices{^i_j})^2-(B\indices{^i_j})^2\right)+\\
&+\sum_i\sum_{j\not=i}-\frac{\sqrt{\delta_i\delta_j}}{1+\sqrt{1-\delta_i}}2\I A\indices{^i_j}B\indices{^i_j}.
\end{align*}
The sum of all these terms is
\begin{align}\label{eq:tr_totale}
\nonumber&\sum_{i}\left(1+\frac{\delta_i}{1+\sqrt{1-\delta_i}}\right)(A\indices{^i_i})^2+\\
\nonumber&+\sum_{i,j\not=i}\left(1+\frac{(\delta_i-\delta_j)^{-1}}{1+\sqrt{1-\delta_i}}2\delta_i\delta_j\right)\left((A\indices{^i_j})^2+(B\indices{^i_j})^2\right)+\\
&+\sum_{i,j\not=i}\frac{\sqrt{\delta_i\delta_j}}{1+\sqrt{1-\delta_i}}\left(1+\frac{\delta_i+\delta_j}{\delta_i-\delta_j}\right)\left((A\indices{^i_j})^2-(B\indices{^i_j})^2\right).
\end{align}
The sum~\eqref{eq:tr_totale} can be reordered as
\begin{align}\label{eq:sum_AB}
\nonumber&\sum_{i,j\not=i}\left(1+\frac{2\delta_i\sqrt{\delta_j}}{1+\sqrt{1-\delta_i}}\frac{\sqrt{\delta_j}+\sqrt{\delta_i}}{\delta_i-\delta_j}\right)(A\indices{^i_j})^2+\\
&+\sum_{i,j\not=i}\left(1+\frac{2\delta_i\sqrt{\delta_j}}{1+\sqrt{1-\delta_i}}\frac{\sqrt{\delta_j}-\sqrt{\delta_i}}{\delta_i-\delta_j}\right)(B\indices{^i_j})^2.
\end{align}
Consider first the~$A$-terms: since~$A$ is symmetric, the coefficient of~$(A\indices{^i_j})^2$ in~\eqref{eq:sum_AB} is
\begin{align*}
&\left(1+\frac{2\delta_i\sqrt{\delta_j}}{1+\sqrt{1-\delta_i}}\frac{\sqrt{\delta_j}+\sqrt{\delta_i}}{\delta_i-\delta_j}\right)+\left(1-\frac{2\delta_j\sqrt{\delta_i}}{1+\sqrt{1-\delta_j}}\frac{\sqrt{\delta_i}+\sqrt{\delta_j}}{\delta_i-\delta_j}\right)=\\
&=2+\frac{2}{\sqrt{\delta_i}-\sqrt{\delta_j}}\left(\sqrt{\delta_j}-\sqrt{\delta_j}\sqrt{1-\delta_i}-\sqrt{\delta_i}+\sqrt{\delta_i}\sqrt{1-\delta_j}\right)=\\
&=2\frac{\sqrt{\delta_i}\sqrt{1-\delta_j}-\sqrt{\delta_j}\sqrt{1-\delta_i}}{\sqrt{\delta_i}-\sqrt{\delta_j}}
\end{align*}
and it is easy to check that this is always non-negative: we can assume that~$\delta_i\geq\delta_j$, so~$\sqrt{1-\delta_j}\geq\sqrt{1-\delta_i}$ and we find
\begin{equation*}
\frac{\sqrt{\delta_i}\sqrt{1-\delta_j}-\sqrt{\delta_j}\sqrt{1-\delta_i}}{\sqrt{\delta_i}-\sqrt{\delta_j}}\geq\sqrt{1-\delta_i}\geq 0.
\end{equation*}
Consider instead the~$B$-terms in~\eqref{eq:sum_AB}: since~$B$ is skew-symmetric, the coefficient of~$(B\indices{^i_j})^2$ is
\begin{align*}
&\left(1-\frac{2\delta_i\sqrt{\delta_j}}{1+\sqrt{1-\delta_i}}\frac{\sqrt{\delta_i}-\sqrt{\delta_j}}{\delta_i-\delta_j}\right)+\left(1-\frac{2\delta_j\sqrt{\delta_i}}{1+\sqrt{1-\delta_j}}\frac{\sqrt{\delta_j}-\sqrt{\delta_i}}{\delta_j-\delta_i}\right)=\\
&=2-\frac{2}{\sqrt{\delta_i}+\sqrt{\delta_j}}\left(\sqrt{\delta_j}\left(1-\sqrt{1-\delta_i}\right)+\sqrt{\delta_i}\left(1-\sqrt{1-\delta_j}\right)\right)=\\
&=\frac{2}{\sqrt{\delta_i}+\sqrt{\delta_j}}\left(\sqrt{\delta_j}\sqrt{1-\delta_i}+\sqrt{\delta_i}\sqrt{1-\delta_j}\right)
\end{align*}
and this is strictly positive. This concludes the proof that~\eqref{eq:sum_AB}, and so also~\eqref{eq:tr_totale} is positive, unless~$R=A+\I B$ vanishes. From the definition of~$R$ we know that this happens only if~$\dot{G}$ is itself identically~$0$.

This shows that, even if~$D\rho\cdot\diff^2_t\delta(\xi\,G_t\,\bar{\xi}\,G_t)$ in~\eqref{eq:diff_seconda_autoval} might be negative, its sum with~$\mrm{Tr}\left(G^{-1}\dot{G}G^{-1}\dot{G}\right)$ is always positive (unless~$\dot{G}=0$). As a consequence, the second variation of~$\widehat{\m{F}}$ along~$u_t$, as written in~\eqref{eq:HK_secondvar}, is positive.

\section{The toric \textit{HK}-energy}\label{sec:toric_HKenergy}

In this Section we prove our variational characterisation and convexity result in the case of toric manifolds, Theorem~\ref{thm:toricHKenergy}, as well as our~$K$-stability result, Theorem~\ref{thm:Kstab}. The key new ingredient, with respect to the periodic case, is an integration by parts formula, Lemma~\ref{lemma:intbyparts_realmm}.

Let~$P$ be a moment polytope, described by the equations in~\eqref{eq:polytope_eq}. Fix a symplectic potential~$u$, let~$G=D^2u$ and assume that~$G\,\xi\,G$ is a smooth matrix-valued function on~$P$, as in~\eqref{eq:def_cs_sympcoord}. To prove our integration by parts formula it will be notationally convenient to let~$\mrm{d}\sigma$ be the measure on~$\diff P$ that on each side~$S_r$ is given by the Lebesgue measure multiplied by~$\abs{\nabla\ell^r}^{-2}$. This is the same measure~$\mrm{d}\sigma$ considered by Donaldson in~\cite{Donaldson_stability_toric}.
\begin{lemma}\label{lemma:intbyparts_realmm}
For a Delzant polytope~$P\subset\bb{R}^n$, fix a symplectic potential~$u$ satisfying Guillemin's boundary conditions and a Higgs term~$\xi$ satisfying the conditions in Proposition~\ref{proposition:toric_boundarycond}. For any function~$f\in\m{C}^0(P)\cap\m{C}^\infty(P^\circ)$ that is either convex or smooth on~$P$ we have
\begin{align*}
&\int_P\mrm{Tr}\left(\left(\mathbbm{1}-\xi\,D^2u\,\bar{\xi}\,D^2u\right)^{\frac{1}{2}}D^2u^{-1}\,D^2f\right)\mrm{d}\mu=\\&=\int_Pf\left(\left(\mathbbm{1}-\xi\,D^2u\,\bar{\xi}\,D^2u\right)^{\frac{1}{2}}D^2u^{-1}\right)^{ab}_{,ab}\mrm{d}\mu+\int_{\diff P}\!f\,\mrm{d}\sigma.
\end{align*}
\end{lemma}
We follow the proof of~\cite[Lemma~$3.3.5$]{Donaldson_stability_toric}.
\begin{proof}
The matrix~$G^{-1}$ is a continuous and bounded function on the whole polytope; by our assumptions it is invertible in~$P^\circ$, and as~$y$ goes to a face~$S_r$,~$G^{-1}(y)$ acquires a kernel containing the vector~$\nabla\ell^r$. In particular,~$G^{-1}$ vanishes at the vertices of~$P$. For~$\delta>0$, let~$P_\delta$ be a polytope contained in~$P$, with faces parallel to those of~$P$ separated by a distance~$\delta$. Inside~$P_\delta$ both~$u$ and~$f$ are smooth, so we can integrate by parts and obtain  
\begin{align}\label{eq:intbyparts_realmm}
\nonumber&\int_{P_\delta}\!\!\!\mrm{Tr}\left(\!\left(\mathbbm{1}-\xi\,G\,\bar{\xi}\,G\right)^{\frac{1}{2}}G^{-1}D^2f\right)\mrm{d}\mu=\int_{P_\delta}\!\!\!f\left(\!\left(\mathbbm{1}-\xi\,G\,\bar{\xi}\,G\right)^{\frac{1}{2}}G^{-1}\right)^{ab}_{,ab}\mrm{d}\mu+\\
&+\int_{\diff P_\delta}\!\!\!\!\!f\,\mrm{div}\left(\!\left(\mathbbm{1}-\xi\,G\,\bar{\xi}\,G\right)^{\frac{1}{2}}G^{-1}\nabla\ell\right)\mrm{d}\sigma-\int_{\diff P_\delta}\!\!\!\left(\!\left(\mathbbm{1}-\xi\,G\,\bar{\xi}\,G\right)^{\frac{1}{2}}G^{-1}\nabla\ell\right)^a\!\diff_af\mrm{d}\sigma
\end{align}
since the outer normal to the boundary is given, at the face~$S_r$, by~$-\frac{\nabla\ell^r}{\abs{\nabla\ell^r}^2}$.

We will take the limit for~$\delta\to 0$ of the two boundary terms on the right-hand side. Fix a face~$S_{\delta,r}$ of~$P_\delta$, parallel to the face~$S_r$ of~$P$; we consider a new system of coordinates, centred at a vertex~$p_0$ of~$S_r$, generated by the vectors~$\set*{\nabla\ell^j\tc j\in I}$, where~$\set*{S_j\tc j\in I}$ is the set of faces that meet at~$p_0$. Up to reordering the faces of~$P$, we can assume that~$r=1$ and that the faces meeting at~$p_0$ are~$S_1,\dots,S_n$. As the distance between~$S_{\delta,r}$ and~$S_{r}$ is~$\delta$, in this new system of coordinates the points of~$S_{\delta,r}$ are described as~$(\delta,y^2,\dots,y^n)$.

In this system of coordinates, the matrix~$G^{-1}$ takes the form
\begin{equation*}
G^{-1}=\begin{pmatrix} y^1+O((y^1)^2) & O(y^1)\\ O(y^1) & \tilde{G}\end{pmatrix}
\end{equation*}
where~$\tilde{G}$ is a~$(n-1)\times(n-1)$ matrix that vanishes at~$0$. The boundary conditions for~$\xi$ imply that there is a smooth matrix-valued function~$\Phi$ such that~$\xi=G^{-1}\,\Phi\,G^{-1}$. Writing the matrix~$\Phi$ as~$\begin{pmatrix}\Phi_{11} & \dots \\\vdots & \tilde{\Phi}\end{pmatrix}$, we can expand~$\left(\mathbbm{1}-\xi\,G\,\bar{\xi}\,G\right)^{\frac{1}{2}}G^{-1}\nabla\ell(p)$ in~$y^1$ as
\begin{align*}
&\left(\mathbbm{1}-G^{-1}\,\Phi\,G^{-1}\,\bar{\Phi}\right)^{\frac{1}{2}}G^{-1}\nabla\ell=\\
&=\left(\begin{pmatrix}\mathbbm{1}& 0\\\vdots & \mathbbm{1}-\tilde{G}\tilde{\Phi}\tilde{G}\bar{\tilde{\Phi}}\end{pmatrix}^{\frac{1}{2}}\!\!\!+O(y^1)\right)\begin{pmatrix}y^1+O((y^1)^2)\\O(y^1)\end{pmatrix}=\begin{pmatrix}y^1+O((y^1)^2)\\O(y^1)\end{pmatrix}
\end{align*}
and the divergence of this vector at a point~$p\in S_{r,\delta}$ is~$1+O(\delta)$. Hence the first boundary term in~\eqref{eq:intbyparts_realmm} goes to~$\int_{\diff P}f\mrm{d}\sigma$ as~$\delta$ goes to~$0$.

As for the second boundary term, it certainly vanishes if~$f$ is smooth up to the boundary of~$P$, as~$G^{-1}\nabla\ell\in O(\delta)$. It is more complicated to obtain the same result for a generic convex function~$f\in\m{C}^0(P)\cap\m{C}^\infty(P^\circ)$, since the gradient of~$f$ might blow up as we go to the boundary of~$P$. However the convexity of~$f$ is sufficient to guarantee that the vanishing of~$G^{-1}\nabla\ell$ will balance out the growth of~$\nabla f$, as shown in~\cite{Donaldson_stability_toric}.

Let~$V$ be the real part of~$\left(\mathbbm{1}-G^{-1}\Phi G^{-1}\bar{\Phi}\right)^{\frac{1}{2}}G^{-1}\nabla\ell$. Since all the other terms in~\eqref{eq:intbyparts_realmm} are real, at least in the limit for~$\delta\to0$, we just have to show that
\begin{equation*}
\int_{\diff P_\delta}\nabla_Vf\mrm{d}\mu\to 0\mbox{ as }\delta\to0.
\end{equation*}
For~$p\in S_{\delta,r}$, let~$q=q(p)$ be the closest point to~$p$ at the intersection between~$\diff P$ and the ray~$p-tV$. Notice that, as~$\nabla\ell^\transpose\cdot V\geq 0$ (see the proof of Theorem~\ref{thm:Kstab} below), the vector~$V$ points inward from~$p$ to~$P_\delta$. We can choose~$\delta_0$ such that, for~$\delta<\delta_0$ and every~$p\in S_{\delta,r}$,~$q(p)$ belongs to the face~$S_r$ of~$\diff P$, as the slopes of~$V$, namely~$V^2/V^1,\dots,V^n/V^1$ are uniformly bounded on~$S_{\delta,r}$.

Let now~$q'=q'(p)=p-(q-p)$, and consider the norm of~$V$ and the distance between~$p$ and~$q$,~$q'$. As~$q$ lies on~$S_r$, we have
\begin{equation*}
\abs{p-q}=\abs{p-q'}=\abs{t V}=(1+O(\delta))\abs{V}
\end{equation*}
and so there is some positive constant~$c$ such that, for~$\delta<\delta_0$
\begin{equation*}
\abs{V}\leq c\abs{p-q}.
\end{equation*}
Using the convexity of~$f$, this implies that at~$p\in S_{\delta,r}$
\begin{equation*}
\begin{split}
\abs*{\nabla_Vf(p)}\leq&\frac{\abs{V}}{\abs{q-p}}\mrm{max}\set*{f(q)-f(p),f(q')-f(p)}\leq\\
\leq&c\,\mrm{max}\set*{f(q)-f(p),f(q')-f(p)}
\end{split}
\end{equation*}
so we can estimate the second boundary term in~\eqref{eq:intbyparts_realmm} as
\begin{align*}
&\abs*{\int_{\diff P_\delta}\nabla_Vf\mrm{d}\mu}\leq\int_{\diff P_\delta}\abs*{\nabla_Vf}\mrm{d}\mu\leq\\ 
&\leq c\int_{p\in\diff P_\delta}\mrm{max}\set*{f(q(p))-f(p),f(q'(p))-f(p)}\mrm{d}\mu\leq\\
&\leq c\,\mrm{Vol}(\diff P_\delta)\max_{p\in\diff P_\delta}\set*{f(q)-f(p),f(q')-f(p)}.
\end{align*}
As~$f$ is uniformly continuous on~$P$ and~$\abs{p-q}\in O(\delta)$, this inequality shows that, as~$\delta$ goes to~$0$,
\begin{equation*}
\int_{\diff P_\delta}\nabla_Vf\mrm{d}\mu\to 0.{}
\end{equation*}\qedhere
\end{proof}
Now the variational characterisation in Theorem~\ref{thm:toricHKenergy} follows easily from Lemma~\ref{lemma:intbyparts_realmm}. Indeed, the Euler-Lagrange equation of the toric \textit{HK}-energy can be readily computed by our integration by parts formula in Lemma~\ref{lemma:intbyparts_realmm}: just notice that the formula is linear in~$f$, so it holds also when~$f$ is a difference of convex functions.  The convexity statement in Theorem~\ref{thm:toricHKenergy} instead follows from the computations in Section~\ref{ConvexitySec} for the periodic \textit{HK}-energy, since the boundary terms in~\eqref{eq:toricHKenergy} are linear in the symplectic potential.

We now proceed to prove Theorem \ref{thm:Kstab}. First, let us recall the definition of uniform~$K$-stability from \cite{Li_uniformKstab}.

\begin{definition}\label{def:Kstab}
Let~$P$ be the moment polytope of a polarized toric manifold~$M$, let~$C=\mrm{d}\sigma(\diff P)/\mrm{d}\mu(P)$ and consider the functional
\begin{equation*}
\m{L}_C(f)=\int_{\diff P}f\,\mrm{d}\sigma-\int_P C\,f\,\mrm{d}\mu.
\end{equation*}
We say that~$M$ is~$K$-stable if~$\m{L}_C(f)\geq 0$ for any convex function~$f\in\m{C}(P)\cap\m{C}^\infty(P^\circ)$, with equality if and only if~$f$ is affine linear.

For a chosen point~$p_0\in P^\circ$, let~$\m{C}_\infty$ be the set of convex functions~$f\in\m{C}(P)\cap\m{C}^\infty(P^\circ)$ such that~$f(p_0)=\mrm{d}f(p_0)=0$. The manifold is \emph{uniformly}~$K$-stable if $\m{L}_c$ vanishes on affine-linear functions and there exists a positive constant~$\lambda$ such that for all~$f\in\m{C}_\infty$
\begin{equation*}
\m{L}_C(f)\geq\lambda\int_{\diff P}f\mrm{d}\sigma.
\end{equation*}
\end{definition}
The following proof is adapted from the analogous results in \cite{Donaldson_stability_toric} and \cite{Li_uniformKstab}, where it is shown that uniform~$K$-stability is a necessary condition for the existence of a solution to Abreu's equation.

\begin{proof}[Proof of Theorem \ref{thm:Kstab}]
Assume that~$(\xi,u)$ is a solution of the toric HcscK system~\eqref{eq:symplectic_HcscK_toric}. From Lemma~\ref{lemma:intbyparts_realmm} we see that, for every convex function~$f\in\m{C}(P)\cap\m{C}^\infty(P^\circ)$
\begin{equation}\label{cor:toric_FutakiEqu}
\int_P\mrm{Tr}\left(\left(\mathbbm{1}-\xi\,D^2u\,\bar{\xi}\,D^2u\right)^{\frac{1}{2}}D^2u^{-1}\,D^2f\right)\mrm{d}\mu=\m{L}_C(f).
\end{equation}
$K$-stability of the toric manifold follows from equation \eqref{cor:toric_FutakiEqu} and the following
\begin{claim}
The matrix~$\left(\mathbbm{1}-\xi\,D^2u\,\bar{\xi}\,D^2u\right)^{\frac{1}{2}}D^2u^{-1}$ is Hermitian and positive definite.
\end{claim}
To prove this claim, it is enough to diagonalise~$D^2u$ and~$\xi\,D^2u\,\bar{\xi}\,D^2u$ as in the proof of Theorem~\ref{thm:convexity}. Then, for the matrix square root we get
\begin{equation*}
\left(\mathbbm{1}-\xi\,D^2u\,\bar{\xi}\,D^2u\right)^{\frac{1}{2}}D^2u^{-1}=Q\left( \mathbbm{1}-\Delta\right)^{\frac{1}{2}}\bar{Q}^\transpose
\end{equation*}
and since the eigenvalues of~$\Delta$ are smaller than~$1$,~$\left(\mathbbm{1}-\Delta\right)^{\frac{1}{2}}$ is positive definite.

To prove that~$M$ is \emph{uniformly}~$K$-stable, we can repeat the argument of \cite[\S$5$]{Li_uniformKstab} almost verbatim. The results in \cite[\S$5$]{Li_uniformKstab} (see in particular the proofs of Lemma $5.1$ and Theorem $4.4$) can be rewritten as follows:
\begin{proposition}[\cite{Li_uniformKstab}]\label{prop:unif_lemma}
Assume that there is a matrix-valued function~$V\in\m{C}^0(P)\cap\m{C}^\infty(P^\circ)$, strictly positive definite in $P^\circ$, such that
\begin{equation}\label{eq:unif_lemma}
\forall f\in\m{C}_\infty,\quad\m{L}_C(f)=\int_PV^{ij}f_{,ij}\mrm{d}\mu.
\end{equation}
Then $P$ is uniformly $K$-stable.
\end{proposition}

Donaldson's integration by parts formula (see \cite[Lemma $3.3.5$]{Donaldson_stability_toric}\footnote{Notice that there is a sign misprint in the statement; the correct sign can be found in \cite[equation $(3.3.6)$]{Donaldson_stability_toric}.}) readily implies that, if there is a symplectic potential $v$ solving Abreu's equation, equality \eqref{eq:unif_lemma} holds for $V=D^2v^{-1}$.

In the present case, we can instead establish \eqref{eq:unif_lemma} using the integration by parts formula of Lemma \ref{lemma:intbyparts_realmm}: equation \eqref{cor:toric_FutakiEqu} implies that $V=\left(\mathbbm{1}-\xi\,D^2u\,\bar{\xi}\,D^2u\right)^{\frac{1}{2}}D^2u^{-1}$ satisfies \eqref{eq:unif_lemma}. As we know that this matrix is positive definite, Proposition \ref{prop:unif_lemma} allows us to conclude.\qedhere
\end{proof}

\section{Solutions of the periodic HcscK system}\label{sec:periodic_sol}
\subsection{Quantitative perturbation}\label{QuantPertSec}

In this Section we will prove our existence result on complex tori, Theorem~\ref{QuantPertThm}. The proof is based on regarding the real moment map equation in~\eqref{thm:HcscKSymplCoordThm} as a perturbation of Abreu's equation~$u^{ab}_{,ab} = A$, by considering the continuity path 
\begin{equation}\label{AbreuEq}
u^{ab}_{,ab} = A_t   
\end{equation}
given for a fixed Higgs tensor~$\xi$ by
\begin{equation}\label{nonlinearLHS}
A_t =t\left(\left(\mathbbm{1}+\left(\mathbbm{1}-\xi\,D^2u\,\bar{\xi}\,D^2u\right)^{\frac{1}{2}}\right)^{-1}\!\!\!\xi\,D^2u\,\bar{\xi}\right)^{ab}_{,ab}\ 
\mbox{for }t\in[0,1]. 
\end{equation}
For~$t=0$,~\eqref{AbreuEq} is just Abreu's equation on the abelian variety, solved by the flat metric. Using~\eqref{eq:HcscKSymplRealEq} instead one can see that~\eqref{AbreuEq} for~$t=1$ becomes the real moment map equation of the Hitchin-cscK system.

We recall a basic estimate on the determinant due to Feng-Sz\'ekelyhidi.
\begin{proposition}[\cite{FengSzekelyhidi_abelian} Lemma 3]\label{DetEstimate} There exist positive constants~$c_1$,~$c_2$, depending only on~$\sup\abs{A}$, such that a solution of~$(u^{ab})_{,ab} = A$ satisfies
\begin{equation*}
0 < c_1 < \det D^2 u < c_2.
\end{equation*}
\end{proposition}
Note that in fact~$c_1, c_2$ can be made explicit in terms of~$\sup\abs{A}$. We apply this with~$A = A_t$ given by~\eqref{nonlinearLHS}. 

In the following we always assume the bound~$\norm*{\xi}_{C^2}<1$.
\begin{lemma}\label{DetEstimateCor}
Suppose~$u = u_t$ solves~\eqref{AbreuEq} for some~$t \in [0,1]$. Then there exist positive constants~$c$,~$c_1$,~$c_2$, such that a bound
\begin{equation*}
\norm{\xi}_{C^2}\norm{u}_{C^4} < c
\end{equation*}  
implies bounds
\begin{equation*}
0 < c_1 < \det D^2 u < c_2.
\end{equation*}
\end{lemma}
\begin{proof} By Proposition~\ref{DetEstimate} we only need to show that having bounds on the eigenvalues of~$\xi\,D^2u\,\bar{\xi}\,D^2u$ away from~$1$ and on~$\norm{\xi}_{C^2}\norm{u}_{C^4}$ imply~$C^0$ bounds on~$A_t$. The required bounds for the eigenvalues~$\delta_i$ away from~$1$ certainly hold provided~$\norm{\xi}^2_{u}$ is bounded away from~$1$, and this condition only depends on a sufficiently small upper bound for~$\norm{\xi}_{C^0}\norm{u}_{C^2}$. Then, if~$\norm{\xi}_{C^0}\norm{u}_{C^2}$ is sufficiently small, a~$C^0$ bound on the term
%\begin{equation*}
$\left(\check{\alpha}\,\xi\,D^2u\,\bar{\xi}\right)^{ab}_{,ab}$
%\end{equation*}
only depends on a bound on~$\norm{\xi}_{C^2}\norm{u}_{C^4}$. As a consequence, a sufficiently small bound on~$\norm{\xi}_{C^2}\norm{u}_{C^4}$ gives a~$C^0$ bound on~$A_t$, as required.\qedhere
\end{proof}
Next, arguing as in~\cite{FengSzekelyhidi_abelian} Section~$4$, we regard~\eqref{AbreuEq}  as a linearised Monge-Amp\`ere equation
\begin{equation}\label{linearMA}
U^{ij} w_{ij} = A_t,
\end{equation}
where~$U^{ij}$ denotes the cofactor matrix of~$D^2 u$, and 
\begin{equation*}
w_{ij} = (\det D^2 u)^{-1}.
\end{equation*}
\begin{lemma} Suppose~$u$ solves the HcscK equation. Then there exist constants~$0<\alpha<1$,~$c_3 > 0$ such that the bound 
\begin{equation*}
\norm{\xi}_{C^2}\norm{u}_{C^4} < c
\end{equation*}  
of Lemma~\ref{DetEstimateCor} implies the bound
\begin{equation*}
\norm{w}_{C^{0,\alpha}} < c_3.
\end{equation*}  
\end{lemma}
\begin{proof}
A sufficiently small bound on~$\norm{\xi}_{C^2}\norm{u}_{C^4}$ clearly gives a bound on~$D^2 u$ and by Lemma~\ref{DetEstimateCor} also on~$(\det D^2 u)^{-1}$, so the eigenvalues of~$D^2 u$ are bounded and bounded away from~$0$. Then the estimates for linearised Monge-Amp\`ere equations in~\cite{TrudingerWang_MongeAmpere} Section~$3.7$ and Corollary~$3.2$, give the required H\"older estimates for a solution of~\eqref{linearMA}, depending only on a~$C^0$ bound on~$A_t$, which is also controlled by~$\norm{\xi}_{C^2}\norm{u}_{C^4}$.\qedhere
\end{proof}
\begin{remark}
On an abelian surface, our equation~\eqref{linearMA} also belongs to the class of two-dimensional linearised Monge-Amp\`ere equations with right hand side in divergence form, which is studied in detail in~\cite{NamLe_estimates}. This work shows that in our argument above the quantity~$\norm{\xi}_{C^2}\norm{u}_{C^4}$ may be replaced with~$\norm{\xi}_{C^1}\norm{u}_{C^3}$, in the two-dimensional case.
\end{remark} 
The subsequent step consists in regarding~$u$ as a solution of the Monge-Amp\`ere equation
\begin{equation*}
 w \det D^2 u = 1.
\end{equation*} 
Caffarelli's Schauder estimate~\cite{Caffarelli_interior} provides the following result.
\begin{lemma}
There exists a constant~$c_4 > 0$ such that the bound   
\begin{equation*}
\norm{\xi}_{C^2}\norm{u}_{C^4} < c
\end{equation*}  
of Lemma~\ref{DetEstimateCor} implies the bound
\begin{equation*}
\norm{u}_{2, \alpha} < c_4.
\end{equation*}   
\end{lemma}  
In turn, standard Schauder estimates applied to the linearised Monge-Amp\`ere equation~\eqref{linearMA} gives a bound on~$\norm{w}_{C^{2, \alpha}}$, from which Caffarelli yields a bound on~$\norm{u}_{4,\alpha}$. The upshot is the following.
\begin{lemma}\label{C4bound} There exists a constant~$c^* > 0$ such that the bound   
\begin{equation*}
\norm{\xi}_{C^2}\norm{u}_{C^4} < c
\end{equation*}  
of Lemma~\ref{DetEstimateCor} implies the bound
\begin{equation*}
\norm{u}_{4, \alpha} < c^*.
\end{equation*} 
\end{lemma}
Suppose now that~$\norm{u_t}_{C^{4, \alpha}}$ blows up along our path
%\begin{equation*}
$(u^{ab}_t)_{,ab} = A_t$.
%\end{equation*} 
Then for some~$\bar{t} \in (0, 1)$ we must have 
%\begin{equation*}
$\norm{u_{\bar{t}}}_{C^{4, \alpha}} = c^*$.
%\end{equation*}
However if the Higgs tensor~$\xi$ satisfies the bound
\begin{equation*}
\norm{\xi}_{C^{2}} < \frac{c}{c^*} 
\end{equation*}
then Lemma~\ref{C4bound} shows that we must also have
%\begin{equation*}
$\norm{u_{\bar{t}}}_{C^{4, \alpha}} < c^*$,
%\end{equation*}
a contradiction. This shows that the set of times~$t \in [0,1]$ for which~\eqref{AbreuEq} is solvable is closed. The equation is solvable for~$t = 0$ with~$u_0 = \frac{1}{2}\abs{y}^2$, so the set is non-empty. Moreover, by choosing~$c$ smaller if necessary, the bound~$\norm{\xi}^2_{u_t} < 1$ holds for all times, so by Theorem~\ref{thm:convexity} a solution for~$t = 1$ is unique if the real and imaginary parts of~$\xi$ are semidefinite. It remains to be seen that the set is also open. As usual this follows from the Implicit Function Theorem. Fix~$t\in[0,1]$ and consider the operator
\begin{equation*}
\mathcal{K}(u) = u^{ab}_{,ab} - A_t
\end{equation*}
with~$A_t$ given by~\eqref{nonlinearLHS}. The differential of this operator was already computed, essentially, in our formula~\eqref{eq:HK_secondvar} for the second variation of the \textit{HK}-energy: consider the Biquard-Gauduchon functional~$\mathcal{H}(u,\xi)$ defined in~\eqref{eq:BGFunc} and the periodic~$K$-energy~$\mathcal{F}(u)$ of~\eqref{eq:KEn}; then~$\mathcal{K}(u)$ satisfies
\begin{equation*}
D\left(\mathcal{F}+t\mathcal{H}\right)_{(u,\xi)}(\varphi)=\frac{1}{2}\int\mathcal{K}_t(u)\,\varphi\,d\mu.
\end{equation*}
So, for the differential of~$\mathcal{K}$ we have
\begin{equation*}
\frac{1}{2}\int \left(D\mathcal{K}_t\right)_u(\varphi)\,\varphi\,d\mu=\frac{d^2}{ds^2}\left(\mathcal{F}+t\mathcal{H}\right)(u+s\,\varphi,\xi)
\end{equation*}
and~$D\mathcal{K}$ is the Hilbert space Hessian of~$\widehat{\m{F}}_t:=\mathcal{F}+t\mathcal{H}$. This is self-adjoint by construction, and computations similar to the proof of Theorem~\ref{thm:convexity} show that it has trivial kernel, for~$t\in[0,1]$. This implies openness and completes the proof of Theorem~\ref{QuantPertThm}.

\subsection{Translation-invariant solutions}\label{1dSec}

In the present Section we shall prove Theorem~\ref{HcscK1dThm}. Thus we consider Higgs tensors~$\xi\in\bb{C}^{2\times 2}$ depending on a single variable, say~$\xi = \xi^{ab}(y^1)$. The complex moment map equation~$\xi^{ab}_{,ab}=0$ reduces to~$\xi^{11}_{11} = 0$, i.e. to the condition~$\xi^{11} = c \in \bb{C}$, by periodicity. With our further assumption~$\det(\xi) = 0$, all possible solutions are given by
\begin{equation}\label{eq:sol_lowrank_complex}
\begin{pmatrix}
0 & 0\\ 0& \xi^{22}(y^1)
\end{pmatrix}
\mbox{ or }
\begin{pmatrix} c & \xi^{12}(y^1)\\ \xi^{12}(y^1) & \frac{\left(\xi^{12}(y^1)\right)^2}{c}\end{pmatrix}.
\end{equation}
We look for solutions~$u(y)$ of the real moment map equation of the form
\begin{equation*}
u(y)=\frac{1}{2}|y|^2+f(y^1)
\end{equation*}
for a periodic function~$f(y^1)$, so 
\begin{equation*}
D^2(u)=\begin{pmatrix}1+f''(y^1) &0\\ 0& 1\end{pmatrix}.
\end{equation*}

\begin{remark}
Once we show that there are such solutions, Theorem~\ref{thm:convexity} will guarantee that in fact all possible solutions~$u(y)$ are of this form. Notice also that in the present two-dimensional situation and under the assumption~$\mathrm{det}(\xi)=0$, the conditions~$\SpecRad(u,\xi)<1$ and~$\norm{\xi}^2<1$ are equivalent. 
\end{remark}

In this low-rank case, the real moment map equation~\eqref{eq:HcscKSymplRealEq} becomes
\begin{equation*}
\left(-(1+f'')+\frac{(1+f'')|\xi^{11}|^2+|\xi^{12}|^2}{1+(1-\norm{\xi}^2_{u})^{1/2}}\right)_{,11}=0.
\end{equation*}
Notice however that if~$\xi$ is of the first type in~\eqref{eq:sol_lowrank_complex} then the equation reduces to~$f^{(4)}=0$, i.e. the metric must have constant coefficients. So it is enough to discuss the case
\begin{equation*}
\xi=\begin{pmatrix} c & F(y^1)\\ F(y^1) & \frac{F(y^1)^2}{c}\end{pmatrix}
\end{equation*}
for some periodic function~$F$. Then the equation becomes
\begin{equation}\label{eq:real_mm_twoderivs}
\left(-(1+f'')+\frac{(1+f'')|c|^2+|F|^2}{1+\left(1-\frac{1}{|c|^2}\left((1+f'')|c|^2+|F|^2\right)^2\right)^{\frac{1}{2}}}\right)''=0
\end{equation}
which is equivalent to the condition 
\begin{equation*}
\frac{(1+f'')|c|^2+|F|^2}{1+\left(1-\frac{1}{|c|^2}\left((1+f'')|c|^2+|F|^2\right)^2\right)^{\frac{1}{2}}}=1+k+f'' 
\end{equation*}
for some real constant~$k$, or equivalently,~$f''$ must satisfy the algebraic equation 
\begin{equation}\label{eq:real_mm_abelian_alt}
\frac{|F|^2}{|c|^2}+\left(f''+1\right)=\frac{2\left(f''+k+1\right)}{|c|^2+\left(f''+k+1\right)^2}.
\end{equation}
When~$F \equiv 0$ a solution~$f''$ is a suitable constant, corresponding to a constant almost-complex structure and a flat metric. We will prove our existence result by perturbing around~$F \equiv 0$ in a quantitative way. It is convenient to introduce the operators 
\begin{equation*}
\begin{split}
P:C^{\infty}(C^1,\bb{R})&\to C^{\infty}_0(S^1,\bb{R})\\
\varphi&\mapsto\varphi'';
\end{split}
\end{equation*}
\begin{equation*}
\begin{split}
Q: C^{\infty}(S^1,\bb{R})\times& C^{\infty}(S^1,\bb{R})\to C^{\infty}(S^1,\bb{R})\\
(\varphi_1,&\varphi_2)\mapsto \frac{(1+\varphi_1)|c|^2+\varphi_2^2}{1+\left(1-\frac{1}{|c|^2}\left((1+\varphi_1)|c|^2+\varphi_2^2\right)^2\right)^{\frac{1}{2}}}-(1+\varphi_1);
\end{split}
\end{equation*}
\begin{equation*}
\begin{split}
R: C^\infty(S^1,\bb{R})\times C^\infty(S^1,\bb{C})&\to C^\infty_0(S^1,\bb{R})\\
(f,F)&\mapsto P\left(Q(P(f),|F|)\right).
\end{split}
\end{equation*}
We will show that if~$|F|\leq|c|<\frac{3}{10}$, then there exists~$f\in C^\infty(S^1,\bb{R})$, unique up to an additive constant, such that
\begin{align*}
&1+f'' >0,\\
&(1+f'')|c|^2+|F|^2 <|c|,\\
&R(f,F) =0.
\end{align*}
Clearly~$P$ is surjective, and its kernel are the constant functions, so it is enough to show that with our assumptions there is a unique~$\phi\in C^\infty_0(S^1,\bb{R})$ satisfying the positivity condition
\begin{equation}\label{positivityIneq}
\phi+1>0,
\end{equation}  
as well as the uniform nonsingularity condition
\begin{equation}\label{nonsingIneq}
\left((1+\phi)|c|^2+|F|^2\right)^2\leq |c|^2-\varepsilon|c|^2
\end{equation}
for some \emph{fixed}~$0 < \varepsilon < 1$, and such that
\begin{equation*}
P(Q(\phi,F))=0.
\end{equation*}
For~$F=0$ we have a unique solution,~$\phi=0$. We consider the continuity path
\begin{equation*}
P(Q(\phi,t\,F))=0\mbox{ for }0\leq t\leq 1,
\end{equation*}
and prove openness and closedness, as usual. 

Consider the differential of the operator
\begin{equation*}
\begin{split}
C^{2,\alpha}_0(S^1)&\to C^{0,\alpha}_0(S^1)\\
\phi&\mapsto P(Q(\phi,F)),
\end{split}
\end{equation*}
namely
\begin{equation*}
\dot{\phi}\mapsto\left[\left(\frac{|c|^2}{\left(1-\frac{1}{|c|^2}p^2\right)^{\frac{1}{2}}\left(1+\left(1-\frac{1}{|c|^2}p^2\right)^{\frac{1}{2}}\right)}-1\right)\dot{\phi}\right]'',
\end{equation*}
where we set~$p = (1+\phi)|c|^2+|F|^2$ for convenience. This is clearly an isomorphism provided its coefficient is bounded away from~$0$. Since we have~$p^2\leq |c|^2-\varepsilon|c|^2$ by~\eqref{nonsingIneq}, a short computation shows that it is enough to assume
\begin{equation*}
|c|^2<\sqrt{\varepsilon}-\varepsilon.
\end{equation*}
What remains to be seen is that the (automatically open) positivity condition~\eqref{positivityIneq} is also closed along the path, while conversely the (automatically closed) uniform nonsingularity condition~\eqref{nonsingIneq} is also open. Then further regularity would follow by a standard bootstrapping argument. Let us prove both claims at once. Our equation holds if and only if there is a constant~$k$ such that
\begin{equation*}
\frac{(1+\phi)|c|^2+|F|^2}{1+\left(1-\frac{1}{|c|^2}\left((1+\phi)|c|^2+|F|^2\right)^2\right)^{\frac{1}{2}}}-(1+\phi)=k.
\end{equation*}
Note that this implies~$1+\phi+k\geq 0$. Moreover, since~$\phi$ has zero average, we have
\begin{equation*}
k+1=\int_0^1\frac{p}{1+\left(1-\frac{1}{|c|^2}p^2\right)^{\frac{1}{2}}} dy^1,\, p = (1+\phi)|c|^2+|F|^2 
\end{equation*}
so we also have~$k+1\geq 0$ and~$k+1\leq p$, hence~$k+1\leq |c|$ along the continuity path. Moreover we have~$1+\phi \geq -k \geq 1-|c|$, so if~$|c| < 1$ then~$1+\phi$ is bounded uniformly away from~$0$. This shows that the positivity condition~\eqref{positivityIneq} is closed along the path provided~$|c| <1$. Now as in~\eqref{eq:real_mm_abelian_alt} we may write our equation as
\begin{equation*}
\left|\frac{F}{c}\right|^2+\phi+1=\frac{2(\phi+1+k)}{|c|^2+(\phi+1+k)^2}.
\end{equation*}
Then~$\phi+1+k$ is a positive solution of the equation
\begin{equation*}
|F|^2-k|c|^2+|c|^2x+\left(\left|\frac{F}{c}\right|^2-k\right)x^2+x^3=2x.
\end{equation*}
Note that all the coefficients on the left hand side are positive. Using the fact that~$-k \geq 1-|c|$, we see that a positive solution of the equation cannot be larger than~$1+|c|$. Hence we find
\begin{equation*}
\phi+1\leq1+|c|-k\leq 2+|c|
\end{equation*}
from which we get, assuming~$|F|^2\leq |c|^2$,
\begin{equation*}
(1+\phi)|c|^2+|F|^2<|c|^2\left(3+ |c|\right).
\end{equation*}
It follows that~\eqref{nonsingIneq} certainly holds for~$0 < \varepsilon < 1$ as long as
\begin{equation*}
|c|^2\left(3+|c|\right)^2\leq 1-\varepsilon
\end{equation*}
hence~$|c|$ must be less or equal than the first positive root of the polynomial
\begin{equation}\label{eq:polinomio_limite}
x^4+6 x^3+9 x^2+\varepsilon -1.
\end{equation}
Under these conditions on~$\varepsilon$,~$|c|$ the uniform nonsingularity condition~\eqref{nonsingIneq} would hold automatically along the path. Finally we need to choose the values of~$|c|$,~$\varepsilon$ in our argument so that all constraints are satisfied. Recall these are
\begin{equation*}
|c| < \left(\sqrt{\varepsilon}-\varepsilon\right)^{\frac{1}{2}},\, |c| < 1
\end{equation*}
as well as the fact that~$|c|$ is less or equal than the first positive root of~\eqref{eq:polinomio_limite}. Direct computation shows that~$\varepsilon= 1/100$,~$|c|<\frac{3}{10}$ is an admissible choice. 

We conclude by examining the integrability condition, as characterised in Proposition~\ref{proposition:integrability_xi}, for the solutions of the HcscK system we have constructed. Since both~$\xi$ and~$D^2u$ depend just on the variable~$y^1$, the integrability condition on~$\xi$ implies  
\begin{equation*}
\partial_1\left(G\xi G\right)_{22}=\partial_2\left(G\xi G\right)_{12}=0.
\end{equation*}
Hence, if~$\xi$ is of the first type in~\eqref{eq:sol_lowrank_complex}, we find~$F'=0$, while for the second type we must have~$F\,F'=0$. In both cases~$\xi$ must be constant. This completes our proof of Theorem~\ref{HcscK1dThm}.

\subsection{Low rank solutions}\label{LowRankSec} 
In this Section we shall prove Theorem~\ref{LowRankThm}. We work with the set of all smooth pairs~$(u, \xi)$, endowed with the topology induced by~$H^{p+2}(\mathbb{T}, \bb{R}) \times H^p(\mathbb{T}, \bb{C})$ for some integer~$p\geq 2$. We will show that, at least nearby a constant~$\xi_0$, lying outside an exceptional subset with empty interior, the locus of solutions to the complex moment map equation~$\xi^{ab}_{,ab} = 0$ satisfying~$\det(\xi) = 0$ is an infinite-dimensional submanifold in the linear space of all solutions to that equation. We claim that this is enough to establish Theorem~\ref{LowRankThm}. Indeed we showed in Sections~\ref{ConvexitySec} and~\ref{QuantPertSec}  that the differential of the operator corresponding to the real moment map equation in~\eqref{eq:HcscKSymplRealEq}, with respect to~$u$, is invertible at any point~$(\tilde{u}, \tilde{\xi})$ that is a solution to the periodic HcscK system. By the Implicit Function Theorem this means that there exists a solution~$u \in H^{p+2}$, which can be expressed locally as a smooth function of~$\xi$, nearby such a~$(\tilde{u}, \tilde{\xi})$. We apply this to the pair~$(u_0, \xi_0)$, and restrict~$\xi$ to lie in the submanifold cut out by~$\det(\xi) = 0$ in the linear space of all solutions to the complex moment map equation. For such smooth~$\xi$, the corresponding~$u$ does not depend on~$p$, by uniqueness, so it is also smooth and Theorem~\ref{LowRankThm} follows.
 
Let us show the claimed submanifold property. As the complex moment map equation is linear, it can be solved by considering the Fourier expansion of~$\xi$,
\begin{equation*}
\xi = \sum_{k \in \bb{Z}^2} \xi_k e^{2\pi i k\cdot x}
\end{equation*}
for constant, complex symmetric matrices~$\xi_k$. Then we have
\begin{equation*}
(\xi^{ab})_{,ab} = (2\pi i)^2\sum_{k \in \bb{Z}^2} \xi^{ab}_k k_{a} k_b e^{2\pi i k\cdot x}
\end{equation*}
so the complex moment map becomes 
\begin{equation*}
k^{T} \xi_k k = 0,\, k\in\bb{Z}^2.
\end{equation*}
Thus solutions can be parametrised as
\begin{align*}
& \xi = \xi_{(0,0)} + \sum_{k_1\neq 0} \left(\begin{matrix}
  0 & \xi^{12}_{(k_1, 0)} \\
  \xi^{12}_{(k_1, 0)}  & \xi^{22}_{(k_1, 0)}  
\end{matrix}\right)e^{2\pi i  k_1 x_1 } + \sum_{k_2\neq 0} \left(\begin{matrix}
  \xi^{11}_{(0,k_2)} & \xi^{12}_{(0, k_2)} \\
  \xi^{12}_{(0, k_2)}  & 0  
\end{matrix}\right)e^{2\pi i k_2 x_2}\\
&+\sum_{k_1 k_2 \neq 0}\left(\begin{matrix}
  \xi^{11}_k & -\frac{1}{2}\left(\xi^{11}_k \frac{k_1}{k_2} + \xi^{22}_k \frac{k_2}{k_1}\right)\\
-\frac{1}{2}\left(\xi^{11}_k \frac{k_1}{k_2} + \xi^{22}_k \frac{k_2}{k_1}\right) & \xi^{22}_k  
\end{matrix}\right) e^{2\pi i k\cdot x}.
\end{align*}
The set of all~$H^p$ solutions~$\mathcal{S}$, with~$p \geq 2$, is an infinite dimensional linear Hilbert manifold, endowed with the smooth function 
%\begin{equation*}
$\det\!: \mathcal{S} \to H^{p-2}(\mathbb{T}, \bb{C})$.
%\end{equation*}
We apply the Implicit Function Theorem to the pair~$(\mathcal{S}, \det)$. The differential of~$\det$ at a solution 
\begin{equation*}
\xi_0 = \left(\begin{matrix}
\alpha & \beta\\
\beta & \gamma
\end{matrix}\right) 
\end{equation*}
is given by
\begin{equation*}
\begin{split}
(d\det)_{\xi_0}(\xi)=& \mrm{Tr} \operatorname{adj}(\xi_0) \xi= \gamma \xi^{11} - 2\beta \xi^{12} +  \alpha \xi^{22}\\
=& \gamma \xi^{11}_{(0,0)} - 2\beta \xi^{12}_{(0,0)} +  \alpha \xi^{22}_{(0,0)}\\
&+ \sum_{k_1\neq 0} \left(-2\beta \xi^{12}_{(k_1, 0)} + \alpha \xi^{22}_{(k_1, 0)}\right) e^{2\pi i  k_1 x_1 }\\
&+ \sum_{k_2\neq 0} \left( 
  \gamma \xi^{11}_{(0,k_2)} -2\beta \xi^{12}_{(0, k_2)}\right)e^{2\pi i k_2 x_2}\\
&+\sum_{k_1 k_2 \neq 0}\left(  \gamma \xi^{11}_k + \beta\left(\xi^{11}_k \frac{k_1}{k_2} + \xi^{22}_k \frac{k_2}{k_1}\right)+\alpha \xi^{22}_k \right) e^{2\pi i k\cdot x}.
\end{split}
\end{equation*}
In order to check if~$\det$ is a submersion we need to see if its differential is onto, i.e. we need to solve the PDE
\begin{equation*}
(d\det)_{\xi_0}(\xi) = f.
\end{equation*}
Suppose now that~$\xi_0$ has constant coefficients. Then this becomes the system of equations, to be solved for Fourier modes~$\xi^{ab}_{k}$, given by
\begin{align}\label{FourierSystem}
&\nonumber\gamma \xi^{11}_{(0,0)} - 2\beta \xi^{12}_{(0,0)} +  \alpha \xi^{22}_{(0,0)} = f_{(0,0)},\\
&\nonumber-2\beta \xi^{12}_{(k_1, 0)} + \alpha \xi^{22}_{(k_1, 0)} = f_{(k_1,0)},\, k_1\neq 0,\\
&\nonumber \gamma \xi^{11}_{(0,k_2)} -2\beta \xi^{12}_{(0, k_2)} = f_{(0,k_2)},\, k_2\neq 0,\\
&  \left(\gamma + \beta  \frac{k_1}{k_2}\right)\xi^{11}_k + \left(\alpha + \beta\frac{k_2}{k_1}\right)\xi^{22}_k = f_{k},\,k_1 k_2 \neq 0.
\end{align}
The differential is onto iff this system can be solved for arbitrary~$f_k$. If this happens, the tangent space~$T_{\xi_0}\mathcal{S}$ is identified with the space of solutions to the corresponding homogeneous system. Thus we are interested in conditions under which the system is solvable. 

If~$\beta \neq 0$, then clearly the first (sets of) equations in~\eqref{FourierSystem} are solvable, and the last is also solvable provided for each~$k$ with~$k_1 k_2 \neq 0$ we do not have both
\begin{equation*}
\gamma = - \frac{k_1}{k_2} \beta,\, \alpha = - \frac{k_2}{k_1}\beta.
\end{equation*}
A sufficient condition is that~$\alpha$,~$\gamma$ are not rational multiples of~$\beta$. This completes the proof of Theorem~\ref{LowRankThm}.

\section{Solutions of the toric HcscK system}\label{sec:toric}
In the present Section we prove our existence result on toric surfaces, Theorem~\ref{QuantPertThmToric}.

Firstly, Lemma~\ref{lemma:complex_mm_sol} can be used to find non-trivial solutions of the complex moment map equation~$\xi^{ab}_{,ab}=0$ also on a toric manifold. If~$T^{abc}$ is skew-symmetric in~$b$,~$c$ and vanishes sufficiently fast at the boundary of~$P$, the matrix
\begin{equation*}
\xi^{ab}=\diff_cT^{abc}+\diff_cT^{bac}
\end{equation*}
solves the moment map equation, and~$D^2u_P\,\xi\,D^2u_P$ is smooth up to the boundary of~$P$. 

Let us digress for a moment to show that a nontrivial solution of the complex moment map equation on a toric manifold corresponds to a non-integrable deformation of the complex structure (see Proposition~\ref{proposition:integrability_xi}). To see this, first notice that computations analogous to the proof of Lemma~\ref{lemma:intbyparts_realmm} also give the following result.
\begin{lemma}\label{lemma:xi_complex_mm_intbyparts}
Let~$P$ be a Delzant polytope, and let~$\xi$ be the representative of a first-order deformation of the complex structure of the corresponding toric manifold. Then for any~$f\in\m{C}^\infty(P)$ we have
\begin{equation*}
\int_P(\xi^{ij})_{,ij}f\mrm{d}\mu=\int_P\xi^{ij}f_{,ij}\mrm{d}\mu.
\end{equation*}
\end{lemma}
Recall that the boundary conditions on~$\xi$ imply that~$\Phi=D^2u\,\xi\,D^2u$ is a matrix-valued function, smooth up to the boundary of~$P$. From Lemma~\ref{lemma:xi_complex_mm_intbyparts} we see that~$\xi$ solves the complex moment map if and only if~$\Phi$ is orthogonal to the space of Hessian matrices with respect to the~$L^2$ product on~$\m{C}^\infty(P)$ defined by~$D^2u$. The integrability condition for~$(u,\xi)$ in~\ref{proposition:integrability_xi} then implies
\begin{corollary}\label{cor:toric_integrability}
A solution of the complex moment map equation~$\xi$ corresponds to an integrable deformation of the complex structure if and only if it vanishes identically.
\end{corollary}
To complete the proof of Theorem~\ref{QuantPertThmToric} we study the real moment map equation via a continuity method, similarly to the periodic case. As uniform $K$-stability implies the existence of a torus-invariant cscK metric, there exists a symplectic potential~$v$ satisfying~$v^{ab}_{,ab} =-C$. Then we can consider the continuity method
\begin{equation*}
u^{ab}_{,ab} =-C+A_t   
\end{equation*}
with~$A_t$ is given by~\eqref{nonlinearLHS}, for a fixed non-zero~$\xi$ and~$t\in[0,1]$. We have a solution for~$t=0$, and openness can be proved as in the periodic case, since the toric \textit{HK}-energy is convex. To prove closedness of the continuity method, we need a priori estimates for the prescribed curvature problem on a toric manifold.
\begin{remark}
The conditions on~$\m{L}_C$ in Definition \ref{def:Kstab} can be stated for~$\m{L}_A$, if~$A$ is a sufficiently regular function on~$P$. In this context, we will say that~$P$ (or~$M$) is \emph{$(A,\lambda)$-stable} if
\begin{equation}\label{eq:properness_Futaki}
\forall f\in\m{C}_\infty,\ \m{L}_A(f):=\int_{\diff P}f\,\mrm{d}\sigma-\int_P A\,f\,\mrm{d}\mu\geq\lambda\int_{\diff P}f\mrm{d}\sigma.
\end{equation}
\end{remark}
Donaldson~\cite{Donaldson_Abreu_IntEst} conjectures that a priori estimates of all orders on~$h = u - u_P$ for solutions of Abreu's equation 
\begin{equation}\label{eq:Abreu}
u^{ab}_{,ab}=-A
\end{equation}
hold, up to the boundary of the polytope, if and only if~$P$ is~$(A,\lambda)$-stable for some positive~$\lambda$. Given such a priori estimates on~$h = u - u_P$, it is straightforward to adapt the argument in the proof of Theorem~\ref{QuantPertThm} showing the closedness of our continuity method for sufficiently small~$\xi$.

Under the additional condition that the function~$A$ in~\eqref{eq:Abreu} is \textit{edge-nonvanishing}, i.e.~$A$ does not vanish identically on any side of the polygon, a proof of Donaldson's conjecture for surfaces has appeared in~\cite{Li_toric_affinegeom}. Assuming this result, it remains to show that the property~\eqref{eq:properness_Futaki} holds along our continuity method, for~$A=C-A_t$. Let us first note that, as in the proof of Lemma~\ref{DetEstimateCor}, a bound on~$\norm{\xi\,D^2u}_{C^2}$ gives a bound on~$A_t$, so if~$\norm{\xi\,D^2u}_{C^2}$ is sufficiently small then~$A$ is edge-nonvanishing. Notice that, even though the~$C^2$-norm of~$u$ blows up at the boundary of~$P$, the boundary conditions for~$\xi$ in Proposition~\ref{proposition:toric_boundarycond} show that~$\xi\,D^2u$ is smooth up to the boundary, so its~$C^2$-norm is always finite.

\begin{lemma}
Let~$P$ be a Delzant polytope, corresponding to a uniformly~$K$-stable toric manifold. Fix a Higgs term~$\xi$ and a symplectic potential~$u$, and let
\begin{equation*}
A=C-t\left(\left(\mathbbm{1}+\left(\mathbbm{1}-\xi\,D^2u\,\bar{\xi}\,D^2u\right)^{\frac{1}{2}}\right)^{-1}\!\!\!\xi\,D^2u\,\bar{\xi}\right)^{ab}_{,ab}.
\end{equation*}
Then~$\m{L}_A$ satisfies~\eqref{eq:properness_Futaki} for a fixed~$\lambda$, independent of~$t$,~$\xi$.
\end{lemma}
\begin{proof}
Choose~$\lambda>0$ such that $\m{L}_C(f)\geq\lambda\int_{\diff P}f\mrm{d}\sigma$ for all~$f\in\m{C}_\infty$. By definition, we have
\begin{equation*}
\m{L}_A(f)=\m{L}_C(f)+t\int_Pf\left(\left(\mathbbm{1}+\left(\mathbbm{1}-\xi\,D^2u\,\bar{\xi}\,D^2u\right)^{\frac{1}{2}}\right)^{-1}\!\!\!\xi\,D^2u\,\bar{\xi}\right)^{ab}_{,ab}\mrm{d}\mu
\end{equation*}
so to prove our claim it is sufficient to show that, for a convex function~$f$, we have
\begin{equation}\label{eq:claim_pos}
\int_Pf\left(\left(\mathbbm{1}+\left(\mathbbm{1}-\xi\,D^2u\,\bar{\xi}\,D^2u\right)^{\frac{1}{2}}\right)^{-1}\!\!\!\xi\,D^2u\,\bar{\xi}\right)^{ab}_{,ab}\mrm{d}\mu\geq0.
\end{equation}
Write~$G=D^2u$ as usual. As in the proof of Theorem~\ref{thm:HcscKSymplCoordThm} (see the discussion after equation~\eqref{eq:HcscKSymplRealEq}) we can rewrite the integrand in~\eqref{eq:claim_pos} using the identity 
\begin{equation*}
\left(\left(\mathbbm{1}+\left(\mathbbm{1}-\xi\,G\,\bar{\xi}\,G\right)^{\frac{1}{2}}\right)^{-1}\!\!\!\xi\,G\,\bar{\xi}\right)^{ab}_{,ab}=G^{ab}_{,ab}-\left(\left(\mathbbm{1}-\xi\,G\,\bar{\xi}\,G\right)^{\frac{1}{2}}\,G^{-1}\right)^{ab}_{,ab}
\end{equation*}
and integrating by parts using Lemma~\ref{lemma:intbyparts_realmm} we obtain
\begin{align*}
&\int_Pf\left(\left(\mathbbm{1}+\left(\mathbbm{1}-\xi\,G\,\bar{\xi}\,G\right)^{\frac{1}{2}}\right)^{-1}\!\!\!\xi\,G\,\bar{\xi}\right)^{ab}_{,ab}\mrm{d}\mu=\\
&=\int_P\mrm{Tr}\left(G^{-1}D^2f\right)\mrm{d}\mu-\int_P\mrm{Tr}\left(\left(\mathbbm{1}-\xi\,G\,\bar{\xi}\,G\right)^{\frac{1}{2}}\,G^{-1}\,D^2f\right)\mrm{d}\mu=\\
&=\int_P\mrm{Tr}\left(\left(G^{-1}-\left(\mathbbm{1}-\xi\,G\,\bar{\xi}\,G\right)^{\frac{1}{2}}\,G^{-1}\right)D^2f\right)\mrm{d}\mu.
\end{align*}
From the proof of Theorem \ref{thm:Kstab}, we know that~$G^{-1}-\left(\mathbbm{1}-\xi\,G\,\bar{\xi}\,G\right)^{\frac{1}{2}}\,G^{-1}$ is positive definite. As~$f\in\m{C}_\infty$ is convex, this proves our claim~\eqref{eq:claim_pos}.\qedhere
\end{proof}

\end{document}